\documentclass[a4paper, 11pt]{article} 
\usepackage[margin=2.4cm, top=2.5cm, bottom=5cm, footskip=3cm]{geometry}	
\usepackage{amsmath}		
\usepackage{amsfonts}		
\usepackage{amsthm}
\usepackage[utf8]{inputenc}
\usepackage{color}
\usepackage{tikz}
\usepackage{slashed}
\usetikzlibrary{calc}
\usetikzlibrary{arrows}  
\usetikzlibrary{decorations.markings}
\usetikzlibrary{shapes,snakes}
\usepackage{hyperref} 
\hypersetup{
	bookmarks=true,
	unicode=false,
	pdftoolbar=true,
	pdfmenubar=true,
	pdffitwindow=false,
	pdfstartview={FitH},
	pdftitle={My title},
	pdfauthor={Author},
	pdfsubject={Subject},
	pdfcreator={Creator},
	pdfproducer={Producer},
	pdfkeywords={keyword1} {key2} {key3},
	pdfnewwindow=true,
	colorlinks=true,   
	linkcolor={blue!80!black}, 
	citecolor={green!50!black}, 
	filecolor={blue!80!black},
	urlcolor={cyan},      
} 
\usepackage[figure]{hypcap}
\usepackage{amssymb}		
\usepackage{amstext}		
\usepackage{graphicx}		
\usepackage[utf8]{inputenc}	
\usepackage[T1]{fontenc}	
\usepackage{pdfpages}		
\usepackage{wrapfig}		
\usepackage{here}			
\usepackage{float}			
\usepackage{fancyhdr}
\usepackage{verbatim}		
\usepackage{wasysym}		
\usepackage{tocloft}		
\usepackage{siunitx}		
\usepackage[arrow, matrix, curve]{xy} 
\usepackage{wrapfig,tabularx,multirow,hyperref,icomma}
\usepackage{mathptmx}
\usepackage{lmodern}
\usepackage{courier}
\usepackage{paralist}
\usepackage{afterpage}
\usetikzlibrary{arrows.meta}
\usepackage{subcaption}
\usetikzlibrary{decorations.pathmorphing,shapes}
\usepackage[section]{placeins}
{\left.\begin{aligned}}
	{\end{aligned}\right\rbrace}
\usepackage{titletoc}
\graphicspath{{figures/}} 

\usepackage{relsize,exscale}

\definecolor{RoyalBlue}{RGB}{35,32, 57}
\definecolor{lightblue}{RGB}{132,131,143}

\tikzstyle{chapterbox} = [fill=black!25!white, thick,
rectangle, inner sep=10pt, inner ysep=5pt]

\tikzstyle{bluebox} = [draw=blue!50!black, fill=blue!10, very thick,
rectangle, inner sep=10pt, inner ysep=5pt] 

\tikzstyle{rahmen blau} = [draw=black, fill=blue!20, very thick,
rectangle, inner sep=10pt, inner ysep=5pt] 

\tikzstyle{graybox} = [draw=black, fill=gray!10, thick,
rectangle, inner sep=10pt, inner ysep=5pt] 

\tikzstyle{box} = [draw=black,  thick,
rectangle, inner sep=10pt, inner ysep=5pt] 

\theoremstyle{definition}
\newtheorem{definition}{Definition}[section]

\newtheorem{theorem}[definition]{Theorem}

\newtheorem{corollary}[definition]{Corollary}

\newtheorem{lemma}[definition]{Lemma}

\theoremstyle{definition}
\newtheorem{remark}[definition]{Remark}

\theoremstyle{definition}

\newcommand{\res}[1]{\raisebox{0.0ex}{\scalebox{1}{$\underset{#1}{\mathrm{Res}}\;$}}}

\renewcommand{\Im}{\mathrm{Im}} 
\renewcommand{\Re}{\mathrm{Re}} 

\newcommand{\z}{\boldsymbol{z}}
\newcommand{\x}{\boldsymbol{x}}

\newcommand{\n}{\boldsymbol{n}}

\newcommand{\iin}{\mathrm{in}}

\newcommand{\aalpha}{\boldsymbol{\alpha}}

\renewcommand{\S}{\mathcal{S}}	
\newcommand{\LHP}{\mathrm{LHP}}
\newcommand{\UHP}{\mathrm{UHP}}

\newcommand\numberthis{\addtocounter{equation}{1}\tag{\theequation}} 
\numberwithin{equation}{section} 

\renewcommand{\a}{\mathfrak{a}}

\title{Diffraction by a Right-Angled No-Contrast Penetrable Wedge: Analytical Continuation of Spectral Functions}  
\date{} 
\author{Valentin D. Kunz\thanks{valentin.kunz@manchester.ac.uk} \ and  \ Raphael C. Assier\thanks{raphael.assier@manchester.ac.uk}  \\ \footnotesize The University of Manchester, Department of Mathematics,  Oxford Road, Manchester, M13 9PL, UK}

\usepackage{xcolor}

\usepackage{lastpage, hyperref}
\hypersetup{
	unicode=false,
	pdftoolbar=true,
	pdfmenubar=true,
	pdffitwindow=false,
	pdfnewwindow=true,
	colorlinks=true,  			
	linkcolor={blue}, 			
	citecolor={black!25!green}, 
	filecolor={blue}
}

\newcommand{\red}{}
\newcommand{\RED}{}
\newcommand{\BLUE}{}
\newcommand{\blue}{}
\newcommand{\Blue}{}
\newcommand{\magenta}{}

\graphicspath{{figures/}}

\begin{document}


	\maketitle

	\begin{abstract} 
		We study the problem of diffraction by a right-angled no-contrast penetrable wedge by means of a two-complex-variable Wiener-Hopf approach. Specifically, the analyticity properties of the unknown (spectral) functions of the two-complex-variable Wiener-Hopf equation are studied. We show that these spectral functions can be analytically continued onto a two-complex dimensional manifold, and unveil their singularities in $\mathbb{C}^2$. To do so, integral representation formulae for the spectral functions are given and thoroughly used. It is shown that the novel concept of additive crossing holds for the penetrable wedge diffraction problem and that we can reformulate the physical diffraction problem as a functional problem using this concept.
	\end{abstract}	
	
	\section{Introduction} 
	For over a century, the canonical problem of diffraction by a penetrable wedge has attracted a great deal of attention in search for a clear analytical solution, which remains an open and challenging problem.
	Beyond their importance to mathematical physics as one of the building blocks of the geometrical theory of diffraction \cite{Keller1962}, wedge diffraction problems also have applications to climate change modelling, as they are related to the scattering of light waves by atmospheric particles such as ice crystals, which is one of the big uncertainties when calculating the Earth's radiation budget (see \cite{Baran2013,GrothEtAl.2015,GrothEtAl.2018} and \cite{SmithEtAl.2015}).
	 
		\red{An important parameter when studying the diffraction by a penetrable wedge is the contrast parameter $\lambda$ which is defined as the ratio of either the electric permittivities $\varepsilon_{1,2}$, magnetic permeabilities $\mu_{1,2}$, or densities $\rho_{1,2}$ corresponding to the material inside and outside the wedge, respectively, depending on the physical context, cf.\! Section \ref{sec:ProblemFormulation}. The case of $\lambda \ll 1$ \BLUE{(high contrast)} is, for instance, studied in \cite{Lyalinov1999} and \cite{Nethercote2020HighContrastApprox}. In \cite{Lyalinov1999}, Lyalinov adapts the Sommerfeld-Malyuzhinets technique to penetrable scatterers and concludes with a far-field approximation taking the geometrical optics components and the diffracted cylindrical waves into account whilst neglecting the lateral waves' contribution. More recently, in \cite{Nethercote2020HighContrastApprox} Nethercote, Assier, and Abrahams provide a method to accurately and rapidly compute the far-field for high-contrast penetrable wedge diffraction taking the lateral waves' contribution into account, by using a combination of the Wiener-Hopf and  Sommerfeld-Malyuzhinets technique. The case of general contrast parameter $\lambda$ is for example considered by \blue{Daniele} and Lombardi  in \cite{DanieleLombardi2011}, which is based on adapting the classical, one-complex-variable Wiener-Hopf technique to penetrable scatterers, and  Salem, Kamel, and Osipov in \cite{SalemEtAl2006}, which is based on an adaptation of the Kontorovich-Lebedev transform to penetrable scatterers. \BLUE{When the wedge has very small opening angle, Budaev and Bogy obtain a convergent Neumann series by using the Sommerfeld-Malyuzhinets technique \cite{Budaev1999}.} \blue{All of these papers offer different ways to numerically compute the total wave-field.} \blue{Another, more theoretically oriented perspective on penetrable wedge diffraction when the wedge is right-angled is provided by Meister, Penzel, Speck, and Teixeira in \cite{MeisterEtAl1994}, which follows an operator-theoretic approach and takes different interface conditions on the two faces of the wedge into account.}}

	\red{The present article studies the case of a no-contrast penetrable wedge. That is, we set $\lambda =1$. Moreover, we assume that the wedge is right-angled. Previous work on this special case includes that of Radlow \cite{Radlow1964PW}, Kraut and Lehmann \cite{KrautLehmann1969}, and Rawlins \cite{Rawlins1977}.  \cite{Radlow1964PW} and  \cite{KrautLehmann1969} are based on a two-complex-variable Wiener-Hopf approach, whereas \cite{Rawlins1977}'s approach is based on Green's functions. 
	\blue{In \cite{Radlow1964PW}} Radlow gives a closed-form solution but it was deemed erroneous by Kraut and Lehmann \cite{KrautLehmann1969} as it led to the wrong corner asymptotics. \BLUE{\Blue{Kraut and Lehmann assume that the wavenumbers inside and outside of the wedge are of similar size, and in \cite{Rawlins1977}, Rawlins extends their work by generalising \cite{KrautLehmann1969}'s scheme to arbitrary opening angles. A description of the diffraction coefficient is given in the right-angled case, in addition to the near-field description provided in \cite{KrautLehmann1969}}. \Blue{Moreover, in \cite{Rawlins1999}, Rawlins obtains the diffraction coefficient for penetrable wedges with arbitrary angles}. Both, \cite{Rawlins1977} and \cite{Rawlins1999} require the wavenumbers inside and outside of the wedge to be of similar size}. \Blue{Another approach on the right-angled no-contrast wedge, that is based on physical optics approximations (also referred to as Kirchhoff approximation), is presented in \cite{GennarelliRiccio2011} and \cite{GennarelliRiccio2012}. These papers modify the ansatz posed in \cite{BurgeEtAl1999}, which extends classical physical optics \cite{Ufimtsev2014} from perfect to penetrable scatterers.} 
	
	Recently, a correction term missing in Radlow's work was given by the authors in \cite{Kunz2021diffraction}. This correction term includes an unknown spectral function and thus, the no-contrast right-angled penetrable wedge diffraction problem remains unsolved.} \blue{The present work is part of an ongoing effort to apply multidimensional complex analysis to diffraction theory  \cite{Kunz2021diffraction,AssierShanin2019,AssierShaninKorolkov2022,Shabat1991,AssierShanin2021VertexGreensFunctions,AssierShanin2021AnalyticalCont,AssierAbrahams2021}. Another approach, also exploiting interesting ideas of multidimensional analysis in the context of wedge diffraction, is given in the monograph \cite{Komech2019} that also contains an excellent review of wedge diffraction problems.} 
	
	We first reformulate the diffraction problem as a two-complex-variable functional problem in the spirit of \cite{AssierShanin2019} and prove that this functional formulation is indeed equivalent to the physical problem. Therefore, solving the functional problem \blue{would} directly \blue{solve} the diffraction problem at hand, which immediately motivates further study of the former. 
	Specifically, we will endeavour to study the analytical continuation of the unknown (spectral) functions of the two-complex-variable Wiener-Hopf equation \eqref{eq.WienerHopf}. Indeed, not only is \red{\blue{the}} knowledge of  the spectral functions' domains of analyticity crucial for completing the classical (one-complex-variable) Wiener-Hopf technique (cf.\!\! \cite{Noble1958}), but by the recent work of Assier, Shanin, and Korolkov \cite{AssierShaninKorolkov2022} we know that knowledge of the spectral functions' singularities allows for computation of the physical fields' far-field asymptotics. \red{Specifically, to obtain closed-form far-field asymptotics of the physical fields, which are represented as inverse double Fourier integrals as given in \eqref{eq.phiscDef} and \eqref{eq.psiDef}, we need to answer the following questions:
			\BLUE{\begin{enumerate}
				\item What are the spectral functions' singularities in $\mathbb{C}^2$?
				\item How can we represent the spectral functions in the vicinity of these singularities?
			\end{enumerate}}
	Addressing these questions, and thereby building the framework that allows us to make further progress, is the main endeavour of the present article.}
	Note that, at this point, it is not clear how the two-complex-variable Wiener-Hopf equation can be solved and generalising the Wiener-Hopf technique to two or more dimensions remains a challenging practical and theoretical task. We refer to the introduction of \cite{AssierShanin2019} for a comprehensive overview of the Wiener-Hopf technique and the difficulties with its generalisation to two-complex-variables.

	The content of the present paper is organised as follows. After formulating the physical problem in Section \ref{subsec:PhysicalProblemFormulation} and rewriting it as a two-complex-variable functional problem involving the unknown spectral functions `$\Psi_{++}$' and `$\Phi_{3/4}$' in Section \ref{subsec:FunctionalProblemFormulation}, the equivalence of these two formulations is proved in Theorem \ref{thm:BackToPhysical}.	Thereafter, in Section \ref{sec:AnalyticalContinuation}, we will study the analytical continuation of the spectral functions in the spirit of \cite{AssierShanin2019}. Using integral representation formulae given in Section \ref{subsec:FormulaeForAnalyticalContinuation}, we unveil the spectral functions' singularities in $\mathbb{C}^2$, \Blue{as well as their local behaviour near those singularities}, in Sections \ref{subsec:FirstStep} and \ref{subsec:SecondStep}. Throughout Sections \ref{sec:ProblemFormulation}--\ref{subsec:SecondStep} we assume positive imaginary part of the wavenumbers $k_1$ and $k_2$, and in Section \ref{subsec:Singularities}, we discuss the spectral functions' singularities on $\mathbb{R}^2$ in the limit $\Im(k_{1,2}) \to 0$. 	Finally, in Section \ref{sec:AdditiveCrossing}, we show that the novel additive crossing property (introduced in \cite{AssierShanin2019}) holds for the spectral function $\Phi_{3/4}$ \red{\BLUE{at intersections of its branch sets. \Blue{This property is critical to obtaining the correct far-field asymptotics (see \cite{AssierShaninKorolkov2022}), and} will lead to the final spectral reformulation of the physical problem, as given in Section \ref{sec:FunctionalProblemReformulated}.}}

	\section{The functional problem for the penetrable wedge}\label{sec:ProblemFormulation}

	\subsection{Formulation of the physical problem}\label{subsec:PhysicalProblemFormulation}
	We are considering the problem of diffraction of a plane wave $\phi_{\iin}$ incident on an infinite, right-angled, penetrable wedge (PW) given by
	\[
	\text{PW} = \{(x_1,x_2) \in \mathbb{R}^2| \ x_1 \geq 0, x_2 \geq 0 \},
	\]
	see Figure \ref{fig:transparentWedge} (left).
	
	\begin{figure}[!h]
		\centering
		\includegraphics[width=\textwidth]{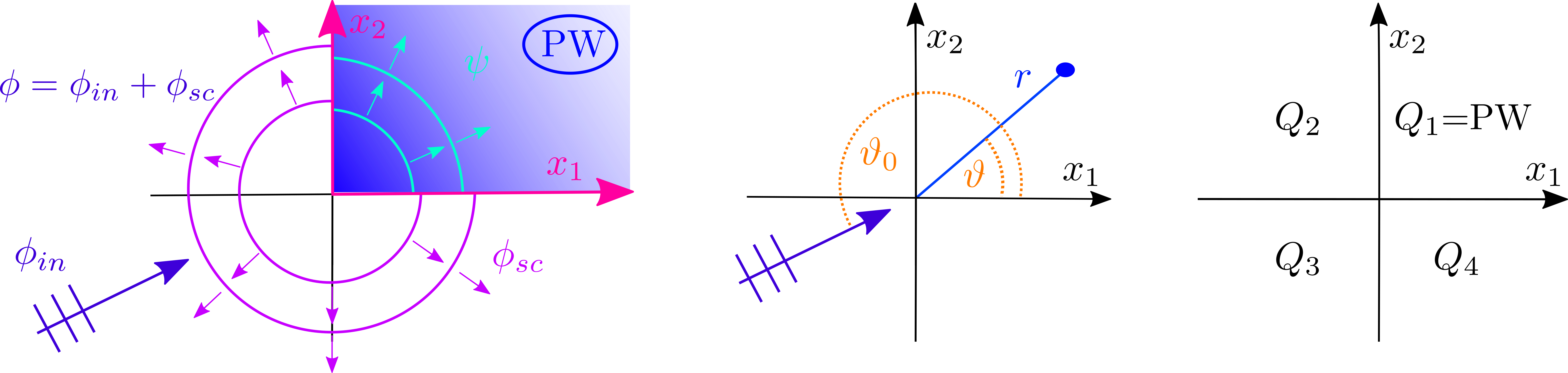}
		\caption{\small Left: Illustration of the problem described by equations \eqref{eq:1.1}--\eqref{eq:1.6}. The scatterer i.e the penetrable wedge is shown in blue with edges in magenta. Middle: Polar coordinate system and incident angle $\vartheta_0$ of $\phi_{\iin}$. Right: Sectors $Q_1,Q_2,Q_2, \text{ and } Q_4$ defined in Section \ref{subsec:FunctionalProblemFormulation}} 
		\label{fig:transparentWedge}
	\end{figure}
	
	We  assume transparency of the wedge and thus expect a scattered field $\phi_{\mathrm{sc}}$ in $\mathbb{R}^2 \setminus \text{PW}$ and a transmitted field $\psi$ in $\text{PW}$. Moreover, we assume  time-harmonicity with the $e^{-i\omega t}$ convention. Therefore, the wave-fields' dynamics are described by two Helmholtz equations, and the incident wave (only supported within $\mathbb{R}^2 \setminus \text{PW}$) is given by
	\[
	\phi_{\iin}(\x) = e^{i \boldsymbol{k}_1 \cdot \x}
	\]	
	where $\boldsymbol{k}_1 \in \mathbb{R}^2$ is the incident wave vector and $\x = (x_1,x_2) \in \mathbb{R}^2$ (this notation will be used throughout the article). Additionally,
	we are describing a \emph{no-contrast penetrable wedge}, meaning that the \emph{contrast parameter} $\lambda$ satisfies 
	\[
	\lambda =1.
	\]
	In the electromagnetic setting, this assumption would correspond to either $\mu_1=\mu_2$ (electric polarisation) or $\epsilon_1=\epsilon_2$ (magnetic polarisation) where $\mu_1$ and $\mu_2$  (resp. $\epsilon_1$ and $\epsilon_2$) are the 	magnetic permeability of the media in $\mathbb{R}^2 \setminus \text{PW}$ and $\text{PW}$  (resp. the electric permittivities of the media in $\mathbb{R}^2 \setminus \text{PW}$ and $\text{PW}$). In the acoustic setting, this assumption corresponds to $\rho_1 = \rho_2$, where $\rho_1$ and $\rho_2$ are the densities of the media (at rest) in $\mathbb{R}^2 \setminus \text{PW}$ and $\text{PW}$, respectively.
	
	Let $k_1=|\boldsymbol{k}_1| =c_1/\omega$ and $k_2=c_2/\omega$ denote the \Blue{wavenumbers} inside and outside  PW, respectively, where $c_1$ and $c_2$ are the wave speeds relative to the media in $\mathbb{R}^2 \setminus \text{PW}$ and $\text{PW}$, respectively. In the electromagnetic setting, $c_j, j=1,2$ corresponds to the speed of light whereas in the acoustic setting, $c_j$ corresponds to the speed of sound. Although $\lambda=1$, the \Blue{wavenumbers}  are different ($k_1 \neq k_2$) since the other media properties defining the speeds of light (in the electromagnetic setting) or sound  (in the acoustic setting) are different.  We refer to \cite{Kunz2021diffraction}, Section 2.1 for a more detailed discussion of the physical context.

	Setting $\phi= \phi_{\mathrm{sc}} + \phi_{\iin}$ (the total wave-field in $\mathbb{R}^2 \setminus \text{PW}$), and letting $\n$ denote the inward pointing normal on $\partial \text{PW}$, the diffraction problem at hand is then described by the following equations.
	
	\begin{alignat}{3}
		\Delta \phi + k^2_1 \phi &=0  &&\text{ in }   \mathbb{R}^2 \setminus \text{PW}, \label{eq:1.1} 	\\
		\Delta \psi + k^2_2 \psi &=0  &&\text{ in }   \text{PW}, \label{eq:1.2} 		\end{alignat}	
	\begin{alignat}{3}			
		\phi&= \psi  &&\text{ on } \partial \text{PW}, \label{eq:1.3} \\
		\partial_{\n} \phi &= \partial_{\n} \psi  &&\text{ on } \partial \text{PW}. \label{eq:1.6}	
	\end{alignat}	
	
	In the electromagnetic setting, $\phi$ and $\psi$ correspond either to the electric
	or magnetic field (depending on the polarization of the incident wave, cf.\!\! \cite{Radlow1964PW,KrautLehmann1969})  in $\mathbb{R}^2 \setminus \text{PW}$ and $\text{PW}$, respectively, whereas in the acoustic setting, $\phi$ and $\psi$ represent the total pressure in  $\mathbb{R}^2 \setminus \text{PW}$ and $\text{PW}$, respectively.

	Equations \eqref{eq:1.1} and \eqref{eq:1.2} are the problem's governing equations, describing the fields' dynamics, whereas the boundary conditions \eqref{eq:1.3}--\eqref{eq:1.6} impose continuity of the fields and their normal derivatives at the wedge's boundary. 
	Introducing polar coordinates $(r,\vartheta)$ (cf.\!\!  Figure \ref{fig:transparentWedge}, middle), we  rewrite the incident wave vector as $\boldsymbol{k}_1 = -k_1(\cos(\vartheta_0),\sin(\vartheta_0))$ where $\vartheta_0$ is the incident angle. The incident wave can then be rewritten as
	\begin{align}
		\phi_{\iin} = e^{-i(\mathfrak{a}_1x_1 + \mathfrak{a}_2x_2)} \label{eq:IncidentRewritten}
	\end{align}
	with \begin{align}
		\mathfrak{a}_1 = k_1 \cos(\vartheta_0) \text{ and } \mathfrak{a}_2 = k_1 \sin(\vartheta_0). \label{eq.a_1,2Definition}
	\end{align} Henceforth, as usual when working in a Wiener-Hopf setting, we assume that the wave numbers have small positive imaginary part $\Im(k_1) = \Im(k_2) >0$ which, since we assumed time harmonicity with the $e^{-i\omega t}$ convention, corresponds to the damping of waves. Note that the imaginary parts of $k_1$ and $k_2$ may be chosen independently of each other, but this does not matter in the present context. In Section \ref{subsec:Singularities}, we will investigate the limit $\Im(k_{1,2}) \to 0$. Moreover, for technical reasons, we have to restrict the incident angle $\vartheta_0 \in (\pi, 3 \pi/2)$, which implies  \[\text{Im}(\mathfrak{a}_{1,2}) \leq - \delta < 0\] for 
			\begin{align}
					\delta = \min\{\Im(k_1)|\cos(\vartheta_0)|, \Im(k_1)|\sin(\vartheta_0)|\}. \label{eq.deltadef}
			\end{align}
	This condition on the incident angle is rather restrictive since it says that the field cannot produce \red{secondary} reflected \red{and transmitted} waves as the incident wave is coming from the $Q_3$-region (see Figure \ref{fig:transparentWedge}, middle and right). However, we \red{will} work around this restriction \blue{in Section \ref{subsec:Singularities}} \red{as long as $\phi_{\iin}$ is incident from within $Q_2 \cup Q_3 \cup Q_4$. \blue{For an incident wave coming} from within $Q_1 = \text{PW}$, the following analysis has to be repeated separately.}
	
	For the problem to be well posed, we also require the fields to satisfy the Sommerfeld radiation condition, meaning that the wave-field should be outgoing in the far-field, and edge conditions called `Meixner conditions', ensuring finiteness of the wave-field's energy near the tip. The radiation condition is \red{\magenta{imposed via the limiting absorption principle on the scattered and transmitted fields: For $\Im(k_{1,2}) >0$, $\phi_{\mathrm{sc}}$ and $\psi$ decay exponentially.}} The edge conditions are given by
	
	\begin{align}
		&	\phi(r, \vartheta) = B +\left(A_1\sin(\vartheta) + B_1\cos(\vartheta)\right)r + \mathcal{O}(r^2) \ \text{as} \ r \to 0, \label{eq.2.56} \\	
		&	\psi(r, \vartheta) = B + \left(A_1\sin(\vartheta) + B_1 \cos(\vartheta)\right)r + \mathcal{O}(r^2) \ \text{as} \ r \to 0 \label{eq.2.57}.
	\end{align}
	We refer to \cite{Kunz2021diffraction} Section 2.1 for a more detailed discussion. Note that \eqref{eq.2.56} and \eqref{eq.2.57} are only valid when $\lambda =1$, and we refer to \cite{Nethercote2020HighContrastApprox} for the general case. Finally, we note that specifying the behaviour of the fields near the wedge's tip and at infinity is required to guarantee unique solvability of the problem described by equations \eqref{eq:1.1}--\eqref{eq:1.6}, see \cite{BabichMokeeva2008}.
	
	\subsection{Formulation as functional problem}\label{subsec:FunctionalProblemFormulation}
	
	Let  $Q_n, \ n=1,2,3,4$ denote the $n$th quadrant of the $(x_1,x_2)$ plane given by
	\begin{align*}
		\text{PW} =	& Q_1 = \{ \x \in \mathbb{R}^2| x_1\geq0, \ x_2\geq0 \}, \	Q_2 = \{ \x \in \mathbb{R}^2| x_1\leq0, \ x_2\geq0 \}, \\
		& 	Q_3 = \{ \x \in \mathbb{R}^2| x_1\leq0, \ x_2\leq0 \}, \ Q_4 = \{ \x \in \mathbb{R}^2| x_1\geq0, \ x_2\leq0 \},
	\end{align*} 
	see Figure \ref{fig:transparentWedge} (right). The  one-quarter Fourier transform of a function $u$ is given by 
	\begin{align}
		U_{1/4}(\boldsymbol{\alpha})=\mathcal{F}_{1/4}[u](\boldsymbol{\alpha}) = \int\hspace{-.2cm}\int_{Q_1} u(\x) e^{i \boldsymbol{\alpha} \cdot \x}
		d\x, \label{def. 1/4FT}
	\end{align}
	and a function's  three-quarter Fourier transform is given by
	\begin{align}
		U_{3/4}(\boldsymbol{\alpha}) = \mathcal{F}_{3/4}[u](\boldsymbol{\alpha}) 
		= \int\hspace{-.2cm}\int_{\cup_{i=2}^4 Q_i} u(\x)e^{i \boldsymbol{\alpha} \cdot \x}
		d\x. \label{def. 3/4FT}
	\end{align}	
	
	Here, we have $\boldsymbol{\alpha} = (\alpha_1,\alpha_2) \in \mathbb{C}^2$ and we write $d\x$ for $dx_1dx_2$. Analysis of where in $\mathbb{C}^2$ the variable $\aalpha$ is permitted to go will be this article's main endeavour. 
	Applying $\mathcal{F}_{1/4}$ to \eqref{eq:1.1} and $\mathcal{F}_{3/4}$ to \eqref{eq:1.2}, using the boundary conditions \eqref{eq:1.3}--\eqref{eq:1.6} and setting
	\begin{alignat*}{3}
		& \Phi_{3/4}(\boldsymbol{\alpha})  =\mathcal{F}_{3/4}[\phi_{\mathrm{sc}}], \   &&\Psi_{++}(\boldsymbol{\alpha})  =\mathcal{F}_{1/4}[\psi], \numberthis \label{eq.SpectralDef}\\ 
		& P_{++}(\boldsymbol{\alpha}) = \frac{1}{(\alpha_1 -\mathfrak{a}_1)(\alpha_2 -\mathfrak{a}_2)}, \quad 	&& K(\boldsymbol{\alpha}) = \frac{k^2_2 - \alpha^2_1 - \alpha^2_2}{k^2_1 - \alpha^2_1 - \alpha^2_2},  \numberthis \label{eq.PandKDefinition}
	\end{alignat*} 
	the following Wiener-Hopf equation is derived after a lengthy but straightforward calculation (see \cite{Kunz2021diffraction} Appendix A):
	\begin{align*}
		- K(\boldsymbol{\alpha})\Psi_{++}(\boldsymbol{\alpha}) = \Phi_{3/4}(\boldsymbol{\alpha}) + P_{++}(\boldsymbol{\alpha}), \label{eq.WienerHopf}
		\numberthis
	\end{align*}	
	which is valid in the product $\S \times \S$ of strips 
	\begin{align}
		\S = \{\alpha \in \mathbb{C}| \ -\varepsilon < \Im(\alpha) < \varepsilon\} 
	\end{align}			
	for 
	\begin{align*}
		& \varepsilon =  \frac{1}{2} \min\{\delta, \Im(k_2)\} \numberthis \label{eq.VarepsilonDef}.
	\end{align*} 
	Here, $\delta$ is as in \eqref{eq.deltadef} and since we chose $\Im(k_1) = \Im(k_2)$, we have, in fact, $\varepsilon = \delta/2$.

	\begin{remark}[Similarity to quarter-plane]\label{rem:SimilarityToQP}
		The Wiener-Hopf equation \eqref{eq.WienerHopf} is formally the same as the Wiener-Hopf equation for the quarter-plane diffraction problem discussed in \cite{AssierShanin2019}. Indeed, the only difference is due to the kernel $K$, which for the quarter-plane is given by $K=1/\sqrt{k^2-\alpha^2_1-\alpha^2_2}$ where $k$ is the (only) wavenumber of the quarter-plane problem (cf.\!\! \cite{Kunz2021diffraction} Remark 2.6). We will encounter this aspect throughout the remainder of the article, and, consequently, most of our formulae and results only differ from those given in \cite{AssierShanin2019} by $K$'s behaviour (its factorisation, see Section \ref{subsec:SomeUsefulFunctions}, and the factorisation's \Blue{domains of analyticity}; see Section \ref{subsec:DomainsForCont}, for instance).  \red{\BLUE{Though the two problems are similar in their spectral formulation, they are very different physically. Indeed, the quarter-plane problem is inherently three-dimensional and its far-field consist\blue{s} of a spherical wave emanating from the corner, some primary and secondary edge diffracted waves as well as a reflected plane wave (see e.g.\! \cite{AssierPeake2012FarField}), while the far-field of the two-dimensional penetrable wedge problem considered here \Blue{consists} of primary and secondary reflected and transmitted plane waves, some cylindrical waves emanating from the corner, as well as some lateral waves.}}
	\end{remark}
	
	\subsubsection{1/4-based and 3/4-based functions}\label{subsec:1/4and3/4based} 
	The two dimensional Wiener-Hopf equation \eqref{eq.WienerHopf} contains two unknown `spectral'  functions, $\Psi_{++}$ and $\Phi_{3/4}$. In the spirit of \cite{AssierShanin2019}, our aim is to convert the physical problem discussed in Section \ref{subsec:PhysicalProblemFormulation} into a formulation in 2D Fourier space, similar to the traditional Wiener-Hopf procedure. For this, the properties of the Wiener-Hopf equation's unknowns are of fundamental importance. Following \cite{AssierShanin2019}, we call these properties 1/4-basedness and 3/4-basedness.
	\begin{definition}
		A function $F(\aalpha)$ in two complex variables is called 1/4-based if there exists a function $f:Q_1 \to \mathbb{C}$ such that 
		\[
		F(\aalpha) = \mathcal{F}_{1/4}[f](\aalpha),
		\]
		and it is called 3/4-based,	 if there exists a function $f:\mathbb{R}^2 \setminus Q_1 \to \mathbb{C}$ such that 
		\[
		F(\aalpha) = \mathcal{F}_{3/4}[f](\aalpha).
		\]		
		Moreover we set for any $x_0 \in \mathbb{R}$
		\begin{align*}
			\UHP(x_0) = \{z \in \mathbb{C}| \ \Im(z) > x_0 \}, \  &   \LHP(x_0) = \{z \in \mathbb{C}| \ \Im(z) < x_0 \},
		\end{align*}
		and		$	\UHP = \UHP(0), \  \LHP = \LHP(0).$
	\end{definition}

	In \cite{Kunz2021diffraction}, it was shown that $\Psi_{++}$ is analytic in $\UHP(\RED{-2\varepsilon})\times \UHP(\RED{-2\varepsilon})$ where $\varepsilon>0$ is as in \eqref{eq.VarepsilonDef}. This is indeed a criterion for 1/4-basedness, that is if a function is analytic in $\UHP \times \UHP$, then it is 1/4-based, see \cite{AssierShanin2019}. However, although a function analytic in $\LHP \times \LHP$ is $3/4$-based, the unknown function $\Phi_{3/4}$ is, in general, not analytic in $\LHP \times \LHP$ (\Blue{see} \cite{AssierShanin2019} and Section \ref{subsec:SecondStep}) and therefore this does not seem to be a criterion useful to diffraction theory. Instead, the correct criterion for the quarter-plane involves the novel concept of additive crossing, see \cite{AssierShanin2019}, and analysing this phenomenon for the penetrable wedge diffraction problem is one of the main endeavours of the present article.

	\begin{remark} \label{rem.PsiBigger}
		Although $\Psi_{++}$ is analytic within $\UHP(-2\varepsilon) \times \UHP(-2\varepsilon)$, we shall, for simplicity, henceforth just work with $\varepsilon$ instead of $2 \varepsilon$. That is, we write that $\Psi_{++}$ is analytic within $\UHP(-\varepsilon) \times \UHP(-\varepsilon)$, bearing in mind that this a priori domain of analyticity can be slightly extended. 
	\end{remark}

	\subsubsection{Asymptotic behaviour of spectral functions}\label{thm:AsymptoticSpectral}
	Before we can reformulate the physical problem of Section \ref{subsec:PhysicalProblemFormulation} as a functional problem similar to the 1D Wiener-Hopf procedure, we require information about the asymptotic behaviour of the unknowns $\Psi_{++}$ and $\Phi_{3/4}$. This will not only be crucial to recover the Meixner conditions, but it will also be of fundamental importance for all of Sections \ref{sec:AnalyticalContinuation} and \ref{sec:AdditiveCrossing}.
	
	In \cite{Kunz2021diffraction} Appendix B it was shown that the spectral functions satisfy the following `spectral edge conditions'.
	For fixed $\alpha^{\star}_2$ (resp. fixed $\alpha^{\star}_1$) in $\UHP(-\varepsilon)$ we have 
	\begin{align}
		&	\Psi_{++}(\alpha_1,\alpha^{\star}_2) = \mathcal{O}(1/|\alpha_1|), \ \text{as} \ |\alpha_1| \to \infty \ \text{in} \ \UHP(\RED{-\varepsilon})  \label{Psi++Behaviour1} \\
		&	\Psi_{++}(\alpha^{\star}_1,\alpha_2) = \mathcal{O}(1/|\alpha_2|), \ \text{as} \ |\alpha_2| \to \infty \ \text{in} \ \UHP(\RED{-\varepsilon}) \label{Psi++Behaviour1.2}
	\end{align}
	and, if neither variable is fixed, 
	\begin{align}
		\Psi_{++}(\alpha_1, \alpha_2) = & \mathcal{O}(1/|\alpha_1||\alpha_2|) \ \text{as} \ 
		|\alpha_1| \to \infty, \ |\alpha_2| \to \infty \ \text{in} \ \UHP(\RED{-\varepsilon}).   \label{Psi++Behaviour2}
	\end{align}	
	Similarly, the function $\Phi_{3/4}$ satisfies the growth estimates \eqref{Psi++Behaviour1}--\eqref{Psi++Behaviour2} as $|\alpha_{1,2}| \to \infty$ in $\S \times \S$.

	\subsubsection{Reformulation of the physical problem}\label{subsec:Reform1st}
	
	Using the results above, we can rewrite the physical problem given by equations \eqref{eq:1.1}--\eqref{eq:1.6} as the following  functional problem. 	
	\begin{definition}\label{def:FirstFunctionalFormulation}
		Let $P_{++}$ and $K$ be as in \eqref{eq.PandKDefinition}.
		We say that two functions $\Psi_{++}$ and $\Phi_{3/4}$ in the variable $\aalpha \in \mathbb{C}^2$ satisfy the `penetrable wedge functional problem' if
		\begin{itemize}
			\item[1.] \ $-K(\aalpha)\Psi_{++}(\aalpha) = \Phi_{3/4}(\aalpha) + P_{++}(\aalpha)$ for all $\aalpha \in \S \times \S;$
			\item[2.] \ $\Psi_{++}$ is analytic in $\UHP(\RED{-\varepsilon}) \times \UHP(\RED{-\varepsilon})$;
			\item[3.] \ $\Phi_{3/4}$ is $3/4$-based;
			\item[4.] \ $\Psi_{++}$ and $\Phi_{3/4}$ satisfy the `spectral edge conditions' given in Section \ref{thm:AsymptoticSpectral}.
		\end{itemize}
	\end{definition}
	The importance of Definition \ref{def:FirstFunctionalFormulation} stems from the following theorem, which proves the equivalence of the penetrable wedge functional problem and the physical problem discussed in Section \ref{subsec:PhysicalProblemFormulation}.
	\begin{theorem}\label{thm:BackToPhysical}
		If  a pair of functions $\Psi_{++}, \Phi_{3/4}$ satisfies the conditions of Definition \ref{def:FirstFunctionalFormulation}, then the functions $\phi_{\mathrm{sc}}$ and $\psi$ given by
		\begin{align}
			\phi_{\mathrm{sc}}(\x)  & = \frac{1}{4\pi^2} \int\hspace{-.2cm}\int_{\mathbb{R}^2} \Phi_{3/4}(\aalpha) e^{-i \aalpha \cdot \x} d\aalpha, \label{eq.phiscDef} \\
			\psi(\x) & = \frac{1}{4\pi^2} \int\hspace{-.2cm}\int_{\mathbb{R}^2} \Psi_{++}(\aalpha) e^{-i \aalpha \cdot \x} d\aalpha, \label{eq.psiDef}
		\end{align}
		satisfy the penetrable wedge problem described by equations \eqref{eq:1.1}--\eqref{eq:1.6} for an incident wave given by $\phi_{\iin} = \exp(-i(\mathfrak{a}_1x_1 + \mathfrak{a}_2x_2))$. 
	\end{theorem}
	
	\begin{proof}
		Since $\Psi_{++}$ is 1/4-based, there exists a function $\psi_{1/4}:\text{PW} \to \mathbb{C}$ with $\mathcal{F}_{1/4}(\psi_{1/4}) = \Psi_{++}$ and similarly, there exists a function $\phi_{3/4}:\mathbb{R}^2 \setminus \text{PW} \to \mathbb{C}$ with $\Phi_{3/4} = \mathcal{F}_{3/4}(\phi_{3/4})$. If we now set $\psi_{1/4} \equiv 0$ on $\mathbb{R}^2 \setminus \text{PW}$ and $\phi_{3/4} \equiv 0$ on $\text{PW}$ we obtain $\Psi_{++} = \mathcal{F}(\psi_{1/4})$ and $\Phi_{3/4} = \mathcal{F}(\phi_{3/4})$ where $\mathcal{F} = \mathcal{F}_{1/4} + \mathcal{F}_{3/4}$ is just the usual 2D Fourier-transform. But then, by uniqueness of the inverse Fourier-transform, we find $\psi = \psi_{1/4}$ and $\phi_{\mathrm{sc}} = \phi_{3/4}$. In particular, $\psi \equiv 0$ on $\mathbb{R}^2 \setminus \text{PW}$ and $\phi_{\mathrm{sc}} \equiv 0$ on PW. 
		Now, by direct calculation using \eqref{eq.WienerHopf}, \eqref{eq.phiscDef}, and \eqref{eq.psiDef} we find
		\begin{align}
			\Delta \phi_{\mathrm{sc}} + k^2_1\phi_{\mathrm{sc}} = \Delta \psi + k^2_2 \psi. \label{eq.Zero=Zero}
		\end{align}
		But since $\psi \equiv 0$ on $\mathbb{R}^2 \setminus \text{PW}$ and $\phi_{\mathrm{sc}} \equiv 0$ on $\text{PW}$, we find that \eqref{eq.Zero=Zero} can only be satisfied if both sides of this equation vanish identically on $\mathbb{R}^2$. In particular
		\begin{align}
			\Delta \phi_{\mathrm{sc}} + k^2_1\phi_{\mathrm{sc}} & =0  \ \text{ in } \mathbb{R}^2 \setminus \text{PW}, \label{eq.RecoveredEq1} \\
			\Delta \psi + k^2_2\psi &=0 \ \text{ in } \text{PW}. \label{eq.RecoveredEq2}
		\end{align}
		Since the incident wave always satisfies \eqref{eq:1.1}, by setting $\phi= \phi_{\iin}+\phi_{\mathrm{sc}}$, we have recovered the Helmholtz equations \eqref{eq:1.1} and \eqref{eq:1.2}. To recover the boundary conditions \eqref{eq:1.3}--\eqref{eq:1.6}, write
		
		\begin{align}
			K(\aalpha) = \frac{K_2(\aalpha)}{K_1(\aalpha)}, \ \text{ for } K_2(\aalpha) = k^2_2 - \alpha^2_1 - \alpha^2_2 \ \text{ and } \ K_1(\aalpha) = k^2_1 - \alpha^2_1 - \alpha^2_2.
		\end{align}
		By a direct computation using Green's theorem (cf.\! \cite{Kunz2021diffraction}), we have 
		\begin{align*}
			\mathcal{F}_{3/4}[\Delta \phi_{\mathrm{sc}} +k^2_1\phi_{\mathrm{sc}}]  = & -K_1(\aalpha)\Phi_{3/4}(\aalpha) 
			+ \int_{\partial \text{PW}} (\partial_{\n} \phi_{\mathrm{sc}}) e^{i \aalpha \cdot \x} dl  
			-  \int_{\partial \text{PW}} \phi_{\mathrm{sc}} \left(\partial_{\n} e^{i \aalpha \cdot \x}\right) dl,  \numberthis \label{eq.PhiBoundaryTerms1} \\
			\mathcal{F}_{1/4}[\Delta \psi + k^2_2\psi] = & -K_2(\aalpha) \Psi_{++}(\aalpha) - \int_{\partial \text{PW}} (\partial_{\n} \psi) e^{i \aalpha \cdot \x} dl  
			+ \int_{\partial \text{PW}} \psi \left(\partial_{\n} e^{i \aalpha \cdot \x}\right) dl .  \numberthis \label{eq.PsiBoundaryTerms1}
		\end{align*}
		And since by  \eqref{eq.RecoveredEq1}--\eqref{eq.RecoveredEq2}
		\begin{align*}
			\mathcal{F}_{3/4}[\Delta \phi_{\mathrm{sc}} +k^2_1\phi_{\mathrm{sc}}] = 	\mathcal{F}_{1/4}[\Delta \psi + k^2_2\psi] =0	
		\end{align*}	
		we have by   \eqref{eq.PhiBoundaryTerms1}--\eqref{eq.PsiBoundaryTerms1}
		\begin{alignat}{3}
			K_1(\aalpha)\Phi_{3/4}(\aalpha) &=   &&\int_{\partial \text{PW}} (\partial_{\n} \phi_{\mathrm{sc}}) e^{i \aalpha \cdot \x} dl  
			-  \int_{\partial \text{PW}} \phi_{\mathrm{sc}} \left(\partial_{\n} e^{i \aalpha \cdot \x}\right) dl,  \label{eq.PhiBoundaryTerms}  \\
			K_2(\aalpha) \Psi_{++}(\aalpha) &=  - &&\int_{\partial \text{PW}} (\partial_{\n} \psi) e^{i \aalpha \cdot \x} dl  
			+ \int_{\partial \text{PW}} \psi \left(\partial_{\n} e^{i \aalpha \cdot \x}\right) dl \label{eq.PsiBoundaryTerms}.
		\end{alignat}	
		Therefore, the Wiener-Hopf equation \eqref{eq.WienerHopf} yields
		\begin{align*}
			-K_1(\aalpha) P_{++}(\aalpha) & = K_1(\aalpha) \Phi_{3/4}(\aalpha) + K_2(\aalpha)\Psi_{++}(\aalpha) \\
			& = \int_{\partial \text{PW}} (\partial_{\n} (\phi_{\mathrm{sc}} - \psi)) e^{i \aalpha \cdot \x} dl  
			-  \int_{\partial \text{PW}} (\phi_{\mathrm{sc}} - \psi) \left(\partial_{\n} e^{i \aalpha \cdot \x}\right) dl. \numberthis \label{eq.PsiPhiBoundaryTerms}
		\end{align*}
		By \cite{Kunz2021diffraction} eq. (A.5)--(A.8), we know 
		\begin{align*}
			-K_1(\aalpha) P_{++}(\aalpha) = & - \int_{\partial \text{PW}} (\partial_{\n} \phi_{\iin}) e^{i \aalpha \cdot \x} dl  
			+ \int_{\partial \text{PW}} \phi_{\iin} \left(\partial_{\n} e^{i \aalpha \cdot \x}\right) dl. \numberthis \label{eq.IncidentBoundaryTerms}
		\end{align*}
		Thus, combining \eqref{eq.PsiPhiBoundaryTerms} and \eqref{eq.IncidentBoundaryTerms}, we have
		\begin{align*}
			0 = & \int_{0}^{\infty} ((\partial_{x_1} \phi)(0^-,x_2) -(\partial_{x_1}\psi)(0^+,x_2)) e^{i \alpha_2 x_2} dx_2+ \int_{0}^{\infty} ((\partial_{x_2} \phi)(x_1,0^-)-(\partial_{x_2}\psi)(x_1,0^+)) e^{i \alpha_1 x_1} dx_1  \\
			& - i \alpha_1 \int_{0}^{\infty} (\phi(0^-,x_2) - \psi(0^+,x_2))e^{i \alpha_2 x_2} dx_2 - i \alpha_2 \int_{0}^{\infty} (\phi(x_1,0^-) - \psi(x_1,0^+) )e^{i \alpha_1 x_1} dx_1,   \numberthis \label{eq.BTFirstStep}
		\end{align*}
		and the proof is complete by Theorem \ref{lem:FromIntegralToInterface}, which implies that each integrand of the integrals in \eqref{eq.BTFirstStep} has to be zero. Note that Theorem \ref{lem:FromIntegralToInterface} can be applied due to the far-field decay of the integrands (cf.\! \cite{Kunz2021diffraction} Section 2.3.3) and the Meixner conditions.
		{\hfill{}}
	\end{proof}
	
	\begin{remark}[Asymptotic behaviour]
		Again, the Sommerfeld radiation condition is satisfied due to the positive imaginary part of $k_{1,2}$ and, by the Abelian theorem (cf.\! \cite{Doetsch1974}), the Meixner conditions hold due to the assumed asymptotic behaviour of the spectral functions.
	\end{remark}

	\section{Analytical continuation of spectral functions}\label{sec:AnalyticalContinuation}
	
	Theorem \ref{thm:BackToPhysical} gives immediate motivation for solving the penetrable wedge functional problem described in Definition \ref{def:FirstFunctionalFormulation}. In the one dimensional case, i.e.\! when solving a one dimensional functional problem by means of the Wiener-Hopf technique, the domains of analyticity of the corresponding unknowns are of fundamental importance \cite{Noble1958}. \red{\BLUE{By \cite{AssierShaninKorolkov2022}, we know that the domains of analyticity of the spectral functions $\Psi_{++}$ and $\Phi_{3/4}$, specifically knowledge of their singularities, are of fundamental importance in the two-complex-variable setting as well. Particularly, knowledge of $\Psi_{++}$ and $\Phi_{3/4}$'s singularity structure in $\mathbb{C}^2$, as well as knowledge of the spectral functions' behaviour in the singularities' vicinity, allows one to obtain closed-form far-field asymptotics of the scattered and transmitted fields, as defined via \eqref{eq.phiscDef}--\eqref{eq.psiDef}. 
	To unveil $\Psi_{++}$ and $\Phi_{3/4}$'s singularity structure, we follow \cite{AssierShanin2019}, wherein the domains of analyticity of the two-complex-variable spectral functions to the quarter-plane problem are studied (which, as mentioned in Remark \ref{rem:SimilarityToQP}, is surprisingly similar to the penetrable wedge problem studied in the present article).
	 }}

	\subsection{Some useful functions}\label{subsec:SomeUsefulFunctions}
	
	Let $\sqrt[\rightarrow]{z}$ denote the square root with branch cut on the positive real axis, and with branch determined by $\sqrt[\rightarrow]{1}=1$ (i.e.\! $\arg(z) \in [0, 2\pi)$). In particular, $\Im(\sqrt[\rightarrow]{z}) \geq 0$ for all $z$ and $\Im(\sqrt[\rightarrow]{z})=0$ if, and only if, $z \in (0,\infty)$. 
	
	As shown in \cite{Kunz2021diffraction}, the kernel $K$ defined in \eqref{eq.PandKDefinition} admits the following  factorisation in the $\alpha_1$-plane
	\begin{align}
		K(\aalpha) = K_{+ \circ}(\aalpha) K_{- \circ}(\aalpha)
	\end{align}
	where
	\begin{align}
		K_{+\circ}(\aalpha) = \frac{\sqrt[\rightarrow]{k^2_2-\alpha_2^2} + \alpha_1}{\sqrt[\rightarrow]{k^2_1-\alpha_2^2} + \alpha_1},  \ 
		K_{-\circ}(\aalpha) = \frac{\sqrt[\rightarrow]{k^2_2-\alpha_2^2} - \alpha_1}{\sqrt[\rightarrow]{k^2_1-\alpha_2^2} - \alpha_1},
	\end{align}
	and the functions $K_{+\circ}$ and $K_{-\circ}$ are analytic in $\UHP(-\varepsilon) \times \S$ and $\LHP(\varepsilon) \times \S$, respectively, for $\varepsilon$ as in \eqref{eq.VarepsilonDef}.\footnote{In \cite{Kunz2021diffraction}, it was proved that $K_{-\circ}$, say, is analytic in $\LHP(-\tilde{\varepsilon}) \times \S$ where $\tilde{\varepsilon} = \min\{\varepsilon, \min_{\alpha}{\sqrt[\rightarrow]{k_1^2-\alpha^2}}\}$ but
	it can be shown that $\tilde{\varepsilon} \geq \varepsilon$ for $\varepsilon$ given by \eqref{eq.VarepsilonDef}.} Analogously, we may choose to factorise in the $\alpha_2$-plane:
	\begin{align}
		K(\aalpha) = K_{\circ +}(\aalpha) K_{\circ -}(\aalpha)
	\end{align}
	where
	\begin{align}
		K_{\circ+}(\aalpha) = \frac{\sqrt[\rightarrow]{k^2_2-\alpha_1^2} + \alpha_2}{\sqrt[\rightarrow]{k^2_1-\alpha_1^2} + \alpha_2},  \ 
		K_{\circ-}(\aalpha) = \frac{\sqrt[\rightarrow]{k^2_2-\alpha_1^2} - \alpha_2}{\sqrt[\rightarrow]{k^2_1-\alpha_1^2} - \alpha_2},
	\end{align}
	where $K_{\circ +}$ and $K_{\circ -}$ are analytic in $\S \times \UHP(-\varepsilon)$ and $\S \times \LHP(\varepsilon)$, respectively. 
	See \cite{Kunz2021diffraction} for a visualisation of $\sqrt[\rightarrow]{z}, \sqrt[\rightarrow]{k^2_{j} - z^2}, \ j=1,2$, and $K_{\pm \circ}$ using phase portraits in the spirit of \cite{Wegert2012}.

	\subsection{Primary formulae for analytical continuation}\label{subsec:FormulaeForAnalyticalContinuation}
	
	Using the kernel's factorisation given in Section \ref{subsec:SomeUsefulFunctions}, we have the following  analytical continuation formulae: 
	\begin{theorem}\label{thm:AnalyticalContinuationPsi++FirstFormulae}
		For $\aalpha \in \S \times \S$,  we have
		\begin{align*}
			\Psi_{++}(\aalpha) =& \frac{1}{4 \pi^2 K_{\circ +}(\aalpha)} \int_{-\infty - i \varepsilon}^{\infty - i \varepsilon} \int_{-\infty + i \varepsilon}^{\infty + i \varepsilon} \frac{K(z_1,z_2) \Psi_{++}(z_1,z_2)}{K_{\circ -}(\alpha_1,z_2)(z_2-\alpha_2)(z_1-\alpha_1)} dz_1 dz_2 \\
			& -  \frac{P_{++}(\aalpha)}{K_{\circ -} (\alpha_1, \mathfrak{a}_2)K_{\circ +}(\aalpha)},
			\numberthis \label{eq.PsiContFormula1} \\
			\Psi_{++}(\aalpha) = & \frac{1}{4 \pi^2 K_{+ \circ}(\aalpha)} \int_{-\infty + i \varepsilon}^{\infty + i \varepsilon} \int_{-\infty - i \varepsilon}^{\infty - i \varepsilon} \frac{K(z_1,z_2) \Psi_{++}(z_1,z_2)}{K_{-\circ }(z_1,\alpha_2)(z_2-\alpha_2)(z_1-\alpha_1)} dz_1 dz_2 \\
			& -  \frac{P_{++}(\aalpha)}{K_{- \circ} (\mathfrak{a}_1, \alpha_2)K_{+\circ}(\aalpha)}.
			\numberthis \label{eq.PsiContFormula2}
		\end{align*}	
		Here, and for the remainder of the article, $\mathfrak{a}_1$ and $\mathfrak{a}_2$ are given by \eqref{eq.a_1,2Definition}.
	\end{theorem}
	The theorem's proof is the exact same as the proof of the analytical continuation formulae for the quarter-plane problem (cf.\! equations $(33)$--$(34)$ and Appendix A in \cite{AssierShanin2019}) and hence omitted. Indeed, we can write $\Phi_{3/4} = \Phi_{-+} + \Phi_{--} + \Phi_{+-}$ for functions $\Phi_{-+}, \ \Phi_{--}$ and $\Phi_{+-}$ analytic in $\LHP(\RED{\varepsilon}) \times \UHP(\RED{-\varepsilon}), \ \LHP(\RED{\varepsilon}) \times \LHP(\RED{\varepsilon})$ and $\UHP(\RED{-\varepsilon}) \times \LHP(\RED{\varepsilon})$, respectively (\Blue{see} \cite{Kunz2021diffraction} eq.\! $(2.21)$), and the analyticity properties of $\Phi_{-+}, \ \Phi_{--}$ and $\Phi_{+-}$ in these domains as well as the analyticity of $\Psi_{++}$ in $\UHP(\RED{-\varepsilon}) \times \UHP(\RED{-\varepsilon})$ are the only key points to finding \eqref{eq.PsiContFormula1} and \eqref{eq.PsiContFormula2}, and these domains agree with those of the quarter-plane problem. 
	The only difference of \eqref{eq.PsiContFormula1}--\eqref{eq.PsiContFormula2} and the corresponding analytical continuation formulae in the quarter-plane problem is due to the difference of the kernel $K$, as discussed in Remark \ref{rem:SimilarityToQP}, and its factorisation.
	
	Observe that the variable on the LHS of \eqref{eq.PsiContFormula1}--\eqref{eq.PsiContFormula2} is $\aalpha \in \S \times \S$. Since all terms involving $\aalpha$ on the equations' RHS are known explicitly, we can choose $\aalpha$ to belong to a domain much larger than $\S \times \S$, thus providing an analytical continuation of $\Psi_{++}$. This procedure will be discussed in the following sections.
	
	\begin{remark}\label{remark:Leray}
		The double integrals in formulae \eqref{eq.PsiContFormula1} and \eqref{eq.PsiContFormula2} can be rewritten as one dimensional Cauchy integrals, as outlined in Appendix \ref{Appendix:Residues}. In the quarter-plane problem discussed in \cite{AssierShanin2019}, such a simplification is not possible as  the singularity of the kernel $K$ is a branch-set, whereas in our case it is just a polar singularity. However, such simplifications of the integral formulae do not significantly simplify the analytical continuation procedure discussed in Sections \ref{subsec:FirstStep} and \ref{subsec:SecondStep} and are hence omitted at this stage.
	\end{remark}
	
	\subsection{Domains for analytical continuation}\label{subsec:DomainsForCont}
	
	For $x_0 \in \mathbb{R}$, let us define the domains $H^+(x_0) \subset \mathbb{C}$ and $H^-(x_0) \subset \mathbb{C}$ as $H^+(x_0)= \UHP(x_0) \setminus \left(h^+_1 \cup h^+_2\right)$ and $H^-(x_0) = \LHP(x_0) \setminus \left(h^-_1 \cup h^-_2\right)$, where $h^+_1, \ h^+_2, \ h^-_1,$ and $h^-_2$ are the curves given by
	\begin{alignat*}{3}
		&	h^-_{j} =  \left\{ - \sqrt[\rightarrow]{k^2_j -x^2}| \ x \in \mathbb{R} \right\}, \ && j=1,2, \\
		&	h^+_j =  \left\{  \sqrt[\rightarrow]{k^2_j -x^2}  | \ x \in \mathbb{R} \right\}, \ && j=1,2.
	\end{alignat*}
	Moreover, we set $H^- = H^-(0), H^+= H^+(0)$,
	see Figure \ref{fig:HHP} for an illustration.
	By the properties of $\sqrt[\rightarrow]{z}$, and due to the positive imaginary part of $k_{1}$ and $k_2$ we indeed have $ h^+_j\in \UHP$ for $j=1,2$ and consequently $h^-_j \in \LHP \ j=1,2$; see \cite{AssierShanin2019}. Moreover, the following holds: 
	\begin{lemma}[\cite{AssierShanin2019}, Lemma 3.2]\label{lem:KappaProperties}
		If $z \in \mathbb{C}\setminus (h^-_1 \cup h^-_2 \cup h^+_1 \cup h^+_2)$, then $\sqrt[\rightarrow]{k^2_j -z^2} \in H^+$ for $j=1,2$.
	\end{lemma}
	
	\begin{figure}[h]
		\centering
		\includegraphics[width=\textwidth]{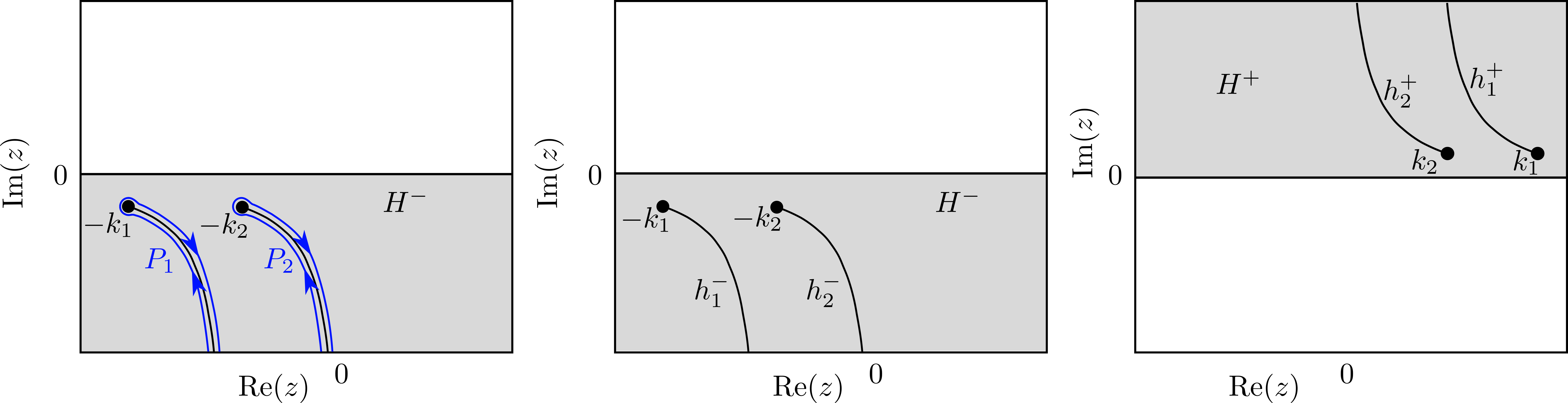}
		\caption{Domains $H^-$ (middle), $H^+$ (right), and contours $P_{1,2}$ (left).}
		\label{fig:HHP}
	\end{figure}
	
	Now, define the contours $P_1$ and $P_2$ as the boundaries of $\mathbb{C} \setminus h^-_1$ and $\mathbb{C} \setminus h^-_2$, see Figure \ref{fig:HHP}. That is, for $j=1,2$, $P_j$ is the contour `starting at $-i\infty$' and moving up along $h^-_j$'s left side, up to $-k_j$, and then moving back towards $-i \infty$ along $h^-_j$'s right side. Intuitively, $P_j$ is just $h^-_j$ but `keeps track' of which side $h^-_j$ was approached from. Set $P = P_1 \cup P_2$.
	
	\begin{remark}
		Formally, all of the following analysis is a priori only valid if the contours $P_1$ and $P_2$ do not cross the points $-k_1$ and $-k_2$, since these are branch points of $\sqrt[\rightarrow]{k^2_1-z^2}$ and $\sqrt[\rightarrow]{k^2_2-z^2}$, so we would have to account for an arbitrarily small circle of radius $b_0$, say, enclosing $-k_1$ and $-k_2$, respectively. However, it is straightforward to show that all formulae remain valid as $b_0  \to 0$, so we do not need to account for this technicality.
	\end{remark}

	\subsection{\normalsize First step of analytical continuation: Analyticity properties of $\Psi_{++}$}\label{subsec:FirstStep}
	
	Since, as already pointed out, the only difference between the present work and \cite{AssierShanin2019} is the structure of the kernel $K$ and, therefore, the structure of the domains $H^{+}$ and $H^-$, most of the following discussion very closely follows \cite{AssierShanin2019}, and we will just sketch most of the details. The reader familiar with \cite{AssierShanin2019} may wish to skip to Theorem \ref{thm:Psi++Domains}. 
	
	Let us first analyse the integral term in \eqref{eq.PsiContFormula1}. As in \cite{AssierShanin2019}, by an application of Stokes' theorem, which tells us that it is possible to deform the surface of integration continuously without changing the value of the integral as long as no singularity is hit during the deformation (cf.\! \cite{Shabat1991}), it is possible to show that:
	\begin{lemma}\label{lem:JStokes}
		For $\aalpha \in \LHP \times \UHP$, the integral term in \eqref{eq.PsiContFormula1} given by
		\begin{align}
			J(\aalpha) = \int_{-\infty - i \varepsilon}^{\infty - i \varepsilon} \int_{-\infty + i \varepsilon}^{\infty + i \varepsilon} \frac{K(z_1,z_2) \Psi_{++}(z_1,z_2)}{K_{\circ -}(\alpha_1,z_2)(z_2-\alpha_2)(z_1-\alpha_1)} dz_1 dz_2 \label{eq:JDefinition01}
		\end{align}
		satisfies 
		\begin{align}
			J(\aalpha) = \int\hspace{-.2cm}\int_{\mathbb{R}^2} \frac{K(z_1,z_2) \Psi_{++}(z_1,z_2)}{K_{\circ -}(\alpha_1,z_2)(z_2-\alpha_2)(z_1-\alpha_1)} dz_1 dz_2. \label{eq:JDefinition}
		\end{align}		
	\end{lemma}
	
	\begin{proof}[Sketch of proof]
		For $\aalpha = (\alpha^{\star}_1, \alpha^{\star}_2) \in \LHP \times \UHP$, the integrand has no singularities in the domain \[\{0<\Im(z_1)<\varepsilon\} \times \{-\varepsilon<\Im(z_2)<0\}\] as can be seen from the properties of $\sqrt[\rightarrow]{z}$ (cf.\!\! Lemma \ref{lem:KappaProperties} and Figure \ref{fig:ContourDeform1st}). Due to the asymptotic behaviour of $\Psi_{++}$, the boundary terms `at infinity' vanish, and therefore an application of Stokes' theorem proves the lemma. The corresponding `contour deformation' is illustrated in Figure \ref{fig:ContourDeform1st}. 
		{\hfill{}}
	\end{proof}
	
	\begin{figure}[!h]
		\centering
		\includegraphics[width=.75\textwidth]{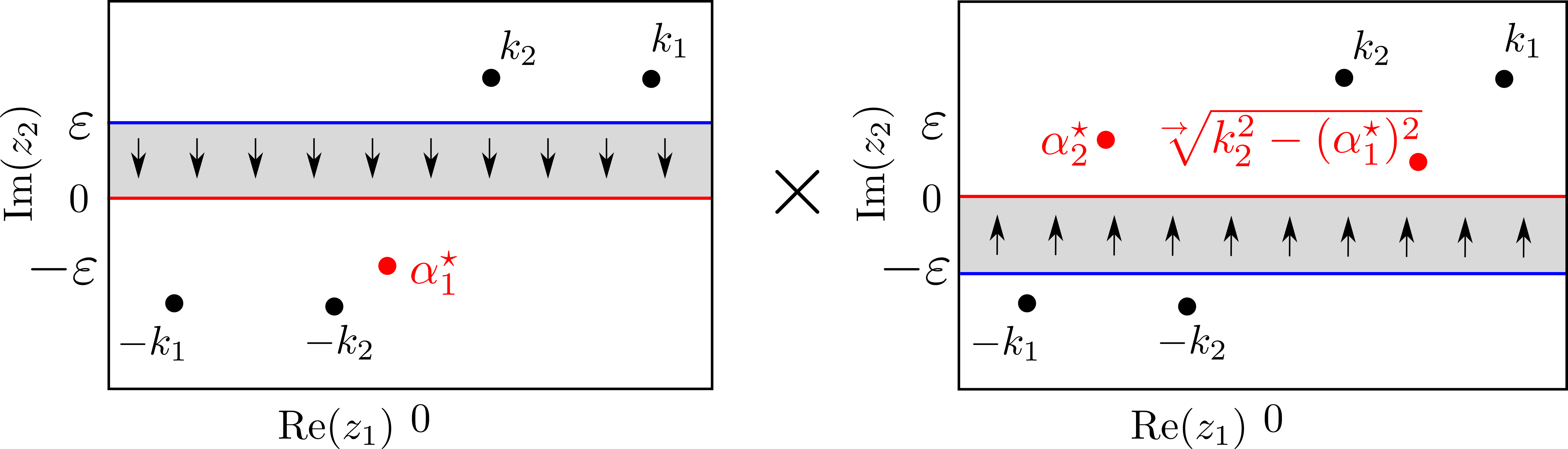}
		\caption{Domain $\{0<\Im(z_1)<\varepsilon\} \times \{-\varepsilon <\Im(z_2)<0\}$. The branch  and polar singularities of the integrand in \eqref{eq:JDefinition01} are shown in black and red, respectively.}
		\label{fig:ContourDeform1st}
	\end{figure}
	
	Henceforth, until specified otherwise, $J$ denotes the function given by \eqref{eq:JDefinition}.
	
	\begin{lemma}\label{lem:JFirstCont}
		$J$ is analytic in $H^- \times \mathrm{UHP}$.
	\end{lemma}
	
	\begin{proof}
		For  any $\z \in \mathbb{R}^2$, we know that the expression \[\frac{K(z_1,z_2) \Psi_{++}(z_1,z_2)}{(z_2-\alpha_2)(z_1-\alpha_1)}\] is analytic in $H^- \times \mathrm{UHP}$ as a function of $\aalpha$ since $\alpha_1$ and $\alpha_2$ are never real (by definition of $\UHP$ and $H^-$, cf.\! Section \ref{subsec:1/4and3/4based} and \ref{subsec:DomainsForCont}), hence the polar factors pose no problem. Let us now investigate 
		\[\frac{1}{K_{\circ -}(\alpha_1,z_2)} = \frac{\sqrt[\rightarrow]{k^2_1-\alpha_1^2} -z_2}{\sqrt[\rightarrow]{k^2_2-\alpha_1^2} -z_2}.\]
		For $\alpha_1 \in H^- $ we know (by Lemma \ref{lem:KappaProperties}) that $\sqrt[\rightarrow]{k^2_2 -\alpha_1^2} \notin \mathbb{R}$.
		Hence, $1/K_{\circ -}(\alpha_1,z_2)$ is analytic in $H^- \times \mathbb{R}$. Let now $\Delta \subset H^-$ be any triangle. Then, using Fubini's theorem (which is possible due to the asymptotic behaviour \eqref{Psi++Behaviour1}--\eqref{Psi++Behaviour2}) we find 
		\begin{align}
			\int_{\partial \Delta} J(w_1,\alpha_2) d w_1 & = 
			\int_{-\infty}^{\infty} \int_{-\infty}^{\infty} \int_{\partial \Delta} \frac{K(z_1,z_2) \Psi_{++}(z_1,z_2)}{K_{\circ -}(w_1,z_2)(z_2-\alpha_2)(z_1-w_1)} dw_1 dz_1 dz_2. 
		\end{align}
		But since the integrand is holomorphic we know, by Cauchy's theorem (\cite{Wegert2012} Theorem 4.2.31), that 
		\begin{align}
			\int_{\partial \Delta} \frac{K(z_1,z_2) \Psi_{++}(z_1,z_2)}{K_{\circ -}(w_1,z_2)(z_2-\alpha_2)(z_1-w_1)} dw_1  =0
		\end{align}	
		and therefore, by Morera's theorem (\cite{Wegert2012} Theorem 4.2.22), we find that $J$ is holomorphic in the first coordinate. Similarly, we find that $J$ is holomorphic in the second coordinate (for $\alpha_2 \in \mathrm{UHP}$) and thus, by Hartogs' theorem (\cite{Shabat1991} Chapter 1, Section 2), we have proved analyticity in $H^- \times \mathrm{UHP}$.	
		{\hfill{}}
	\end{proof}
	
	We now want to investigate the behaviour of $J$ on the boundary of $H^- \times \UHP$. As in Lemma \ref{lem:JStokes}  it can be shown that 
	\begin{align}
		J(\aalpha) = \int_{-\infty}^{\infty } \int_{-\infty + i b_0}^{\infty + i b_0} \frac{K(z_1,z_2) \Psi_{++}(z_1,z_2)}{K_{\circ -}(\alpha_1,z_2)(z_2-\alpha_2)(z_1-\alpha_1)} dz_1 dz_2
	\end{align}
	and that for sufficiently small $ b_0 \in  (0, \varepsilon)$, for all $\z \in (- \infty + i b_0, \infty + i b_0) \times \mathbb{R}$ the integrand \[(K(z_1,z_2) \Psi_{++}(z_1,z_2))/(K_{\circ -}(\alpha_1,z_2)(z_2-\alpha_2)(z_1-\alpha_1))\] is analytic, as a function of $\aalpha$, in a sufficiently small neighbourhood of any fixed $(\alpha^{\star}_1,\alpha^{\star}_2) \in \mathbb{R} \times \UHP$. Just as in the proof of Lemma \ref{lem:JFirstCont}, this yields:
	
	\begin{lemma}\label{lem:JFirstBoundaryCont}
		$J$ can be analytically continued onto $\mathbb{R} \times \mathrm{UHP}$.
	\end{lemma}
	
	Similarly (again, see \cite{AssierShanin2019} for the technical details involved):
	
	\begin{lemma}\label{lem:JSecondBoundaryCont}
		$J$ can be analytically continued onto the other boundary components $H^- \times \mathbb{R}$, $(P\setminus \{-k_1,-k_2\}) \times \UHP$, and $J$ is continuous on $P \times \mathbb{R}$. 
	\end{lemma}	
	
	Let us discuss the remaining terms involved in \eqref{eq.PsiContFormula1}, and recall that $\mathfrak{a}_{1,2}$ are given by \eqref{eq.a_1,2Definition}. By definition of $H^-$, we find that the external terms  $1/K_{\circ +}$ and $1/K_{\circ -}(\alpha_1,\mathfrak{a}_2)$   are  analytic in $H^- \times \mathrm{UHP}$. $P_{++}$ however, has a simple pole at $\mathfrak{a}_1$ and is therefore only analytic in $H^- \setminus \{\mathfrak{a}_1\} \times \mathrm{UHP}$. Analyticity of these terms on the boundary elements $\mathbb{R} \times \mathrm{UHP}, \ H^- \setminus \{\mathfrak{a}_1\} \times \mathbb{R}, \ P\setminus \{-k_{1},-k_2\} \times \mathrm{UHP}, \ (P \setminus \{-k_{1},-k_2\} \times \mathbb{R}) \setminus \{(\alpha_1, -\sqrt[\rightarrow]{k^2_2-\alpha_1^2})| \ \alpha_1 \in P_2 \}$  follows by definition of these sets and the properties of $\sqrt[\rightarrow]{z}$. Observe that we have to exclude  $\{(\alpha_1, -\sqrt[\rightarrow]{k^2_2-\alpha_1^2})| \ \alpha_1 \in P_2 \}$ from  $P\setminus \{-k_{1},-k_2\} \times \mathbb{R}$ since it is a polar singularity of the external factor $1/K_{\circ +}$. 
	This is different from the quarter-plane problem. To summarise:
	
	\begin{corollary}
		The $1/4$-based spectral function $\Psi_{++}$ satisfying the penetrable wedge functional problem \ref{def:FirstFunctionalFormulation} can be analytically continued onto $H^- \setminus \{\mathfrak{a}_1\} \times \mathrm{UHP}$. It can moreover be analytically continued onto the boundary elements $\mathbb{R} \times \mathrm{UHP}, \ H^- \setminus \{\mathfrak{a}_1\} \times \mathbb{R}, \ P\setminus \{-k_{1},-k_2\} \times \mathrm{UHP}$, and continuously continued onto $(P \times \mathbb{R}) \setminus \{(\alpha_1, -\sqrt[\rightarrow]{k^2_2-\alpha_1^2})| \ \alpha_1 \in P_2 \}$.
	\end{corollary}
	
	Repeating the above procedure but using \eqref{eq.PsiContFormula2} instead (again, see \cite{AssierShanin2019} for the technical details involved), we obtain:
	
	\begin{corollary}
		$\Psi_{++}$ can be analytically continued onto $ \mathrm{UHP} \times H^- \setminus \{\mathfrak{a}_2\}$ and onto the boundary elements $\mathrm{UHP} \times \mathbb{R}, \ \mathbb{R} \times H^- \setminus \{\mathfrak{a}_2\},  \mathrm{UHP} \times P\setminus \{-k_{1},-k_2\}$, and continuously continued onto $(\mathbb{R} \times P) \setminus \{( -\sqrt[\rightarrow]{k^2_2-\alpha_2^2}, \alpha_2)| \ \alpha_2 \in P_2 \}$.
	\end{corollary}
	
	Therefore:
	
	\begin{theorem}\label{thm:Psi++Domains}
		$\Psi_{++}$ can be analytically continued onto  
		\[
		\left(\UHP \times \mathbb{C} \setminus \left(h^-_1 \cup h^-_2 \cup \{\mathfrak{a}_1\}\right) \right) \cup \left( \mathbb{C} \setminus \left(h^-_1 \cup h^-_2 \cup \{\mathfrak{a}_2\}\right) \times \UHP\right),
		\]
		as shown in Figure \ref{fig:Domain++}, and $\Psi_{++}$ is analytic on this domain's boundary except on the distinct boundary $(P \times \mathbb{R}) \cup (\mathbb{R} \times P)$ on which $\Psi_{++}$ is continuous everywhere except on the curves given by $\{(\alpha_1, -\sqrt[\rightarrow]{k^2_2 - \alpha^2_1})| \alpha_1 \in P_2\}$  and $\{( -\sqrt[\rightarrow]{k^2_2 - \alpha^2_2}, \alpha_2)| \alpha_2 \in P_2\}$ which yield polar singularities.
	\end{theorem}

	\begin{figure}[h]
		\includegraphics[width=\textwidth]{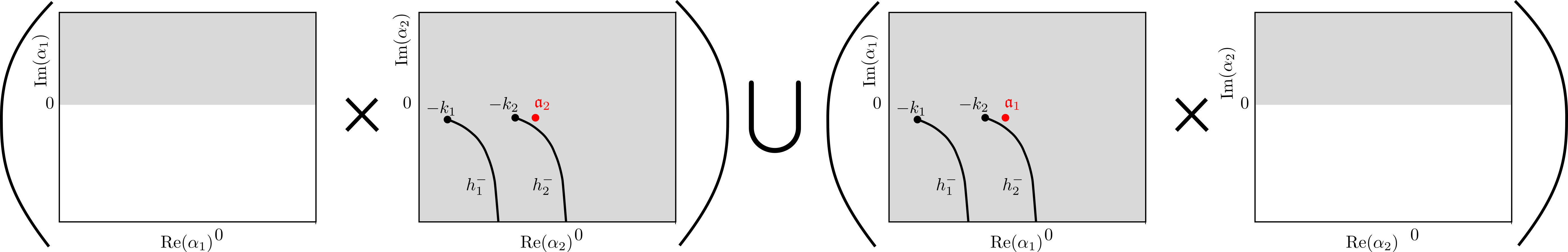}
		\caption{Domain of analyticity of $\Psi_{++}$. Polar singularities $\alpha_2 \equiv \mathfrak{a}_2$ and $\alpha_1 \equiv \mathfrak{a}_1$ are shown in red whereas the branch lines $h^-_1$ and $h^-_2$ are shown in black.}
		\label{fig:Domain++}
	\end{figure}
	
	Naturally, we ask whether a formula can be found for $\Psi_{++}$ in the `missing' parts of Figure \ref{fig:Domain++} that is, whether we can find a formula for $\Psi_{++}$ in $\LHP \times \LHP$. Moreover, from \cite{AssierShanin2019}, we \Blue{anticipate} that the study of $\Psi_{++}$ in $\LHP \times \LHP$ is directly linked to finding a criterion for 3/4-basedness of $\Phi_{3/4}$. This will be the topic of the following sections.
	
	\subsection{\normalsize  Second step of analytical continuation: Analyticity properties of $\Phi_{3/4}$}\label{subsec:SecondStep}
	
	\begin{lemma}\label{lemma:formulaPsi01SecondStep} \
		Let $\varepsilon>b_0>0$, for $\varepsilon$ as in \eqref{eq.VarepsilonDef}. Then	$\Psi_{++}$ satisfies
		\begin{align*}
			\Psi_{++}=& \frac{1}{4 \pi^2 K_{\circ+}} \int_{P} \int_{\mathbb{R}-ib_0} \frac{K(z_1,z_2) \Psi_{++}(z_1,z_2)}{K_{\circ-}(\alpha_1,z_2)(z_2-\alpha_2)(z_1-\alpha_1)} dz_1 dz_2 \\
			& - \frac{K_{-\circ}(\alpha_1,\mathfrak{a}_2)}{K_{\circ+} K_{\circ-}(\alpha_1,\mathfrak{a}_2)K_{-\circ}(\mathfrak{a}_1,\mathfrak{a}_2) (\alpha_1 -\mathfrak{a}_1)(\alpha_2-\mathfrak{a}_2)} \numberthis \label{eq.PsiCont}
		\end{align*}
		where the RHS of \eqref{eq.PsiCont} is defined on $H^-(-b_0) \times \mathrm{UHP}$.	
	\end{lemma}
	
	Our proof is slightly different from \cite{AssierShanin2019}, as the integral in \eqref{eq.PsiCont} is not improper.
	
	\begin{proof}
		Again, we first focus on 
		\begin{align*}
			J(\aalpha) = \int\hspace{-.2cm}\int_{\mathbb{R}^2} \frac{K(z_1,z_2)\Psi_{++}(z_1,z_2)}{K_{\circ-}(\alpha_1,z_2)(z_2-\alpha_2)(z_1-\alpha_1)} dz_1 dz_2.
		\end{align*}
		Change the $z_1$ contour from $\mathbb{R}$ to $\mathbb{R}-ib_0$, which will not hit any singularities of the integrand and therefore
		\begin{align}
			J(\aalpha) = \int_{\mathbb{R}} \int_{\mathbb{R}-ib_0} \frac{K(z_1,z_2)\Psi_{++}(z_1,z_2)}{K_{\circ-}(\alpha_1,z_2)(z_2-\alpha_2)(z_1-\alpha_1)} dz_1 dz_2.
		\end{align}	
		Now, change the $z_2$ contour from $\mathbb{R}$ to $P$. This will only hit the singularity of the integrand at $z_2 = \mathfrak{a}_2$ and therefore, this picks up a residue of the integrand at $z_2 = \mathfrak{a}_2$ (relative to clockwise orientation). Indeed,  $K\Psi_{++}$ has no singularities on $(\mathbb{R} - i b_0) \times H^-$ and no singularities on $(\mathbb{R} - i b_0) \times P$, as can be seen from formula \eqref{eq.PsiContFormula2}. 
		Therefore, we obtain 
		\begin{align*}
			J(\aalpha) =& \int_{P} \int_{\mathbb{R}-ib_0} \frac{K(z_1,z_2)\Psi_{++}(z_1,z_2)}{K_{\circ-}(\alpha_1,z_2)(z_2-\alpha_2)(z_1-\alpha_1)} dz_1 dz_2  \\
			&-
			2 \pi i \int_{\mathbb{R}-ib_0}\res{z_2 = \mathfrak{a}_2} \left(\frac{K(z_1,z_2)\Psi_{++}(z_1,z_2)}{K_{\circ-}(\alpha_1,z_2)(z_2-\alpha_2)(z_1-\alpha_1)}\right) dz_1. \numberthis
		\end{align*}
		The remainder of the proof is identical to \cite{AssierShanin2019}; that is, compute the residue using formula \eqref{eq.PsiContFormula2} where only the external additive term in \eqref{eq.PsiContFormula2} contributes to the residue, and identify it as the minus-part of a Cauchy sum-split which can be explicitly computed by pole-removal. Use the resulting formula for $J$ in \eqref{eq.PsiContFormula1} to  obtain \eqref{eq.PsiCont}.
		{\hfill{}}
	\end{proof}
	
	Similarly, changing the roles of \eqref{eq.PsiContFormula1} and \eqref{eq.PsiContFormula2} in the previous proof, we find:
	\begin{lemma}
		Let $\varepsilon>b_0>0$. Then $\Psi_{++}$ satisfies
		\begin{align*}
			\Psi_{++}=& \frac{1}{4 \pi^2 K_{+ \circ}} \int_{\mathbb{R}-ib_0} \int_{P} \frac{K(z_1,z_2) \Psi_{++}(z_1,z_2)}{K_{-\circ}(z_1,\alpha_2)(z_2-\alpha_2)(z_1-\alpha_1)} dz_1 dz_2 \\
			& - \frac{K_{\circ-}(\mathfrak{a}_1,\alpha_2)}{K_{+\circ} K_{-\circ}(\mathfrak{a}_1,\alpha_2)K_{\circ-}(\mathfrak{a}_1,\mathfrak{a}_2) (\alpha_1 -\mathfrak{a}_1)(\alpha_2-\mathfrak{a}_2)} \numberthis \label{eq.PsiCont2}
		\end{align*}
		where the RHS of \eqref{eq.PsiCont2} is defined for  $\aalpha \in \mathrm{UHP} \times H^-(-b_0)$.
	\end{lemma}
	
	We are now ready to prove this section's main result.
	
	\begin{theorem}\label{PhiCont}
		The function $\Phi= K \Psi_{++}$ can be analytically continued onto $\left(H^-(\varepsilon) \setminus \{\mathfrak{a}_1\}\right) \times \left(H^-(\varepsilon) \setminus \{\mathfrak{a}_2\}\right)$, for $\mathfrak{a}_{1,2}$ given by \eqref{eq.a_1,2Definition} and $\varepsilon$ as in \eqref{eq.VarepsilonDef}, and $\Phi$ is continuous on $P \times P$. The residues of $\Phi$ near the polar singularities $\alpha_1\equiv\mathfrak{a}_1$ and $\alpha_2\equiv\mathfrak{a}_2$ are given by
		\begin{align}
			\res{\alpha_1=\mathfrak{a}_1}{\Phi} &=  - \frac{K_{\circ-}(\mathfrak{a}_1,\alpha_2) }{ K_{\circ-}(\mathfrak{a}_1,\mathfrak{a}_2)(\alpha_2-\mathfrak{a}_2)}, \label{eq.PhiRes1}\\
			\res{\alpha_2=\mathfrak{a}_2}{\Phi} &=  - \frac{K_{-\circ}(\alpha_1,\mathfrak{a}_2)}{ K_{-\circ}(\mathfrak{a}_1,\mathfrak{a}_2)(\alpha_1 -\mathfrak{a}_1)}. \label{eq.PhiRes2}
		\end{align}
	\end{theorem}
	
	\begin{proof}
		Due to \eqref{eq.PsiCont} we can, for $\boldsymbol{\alpha} \in H^-(-b_0) \times \S$, write
		\begin{align*}
			\Phi=& \frac{K_{\circ-}}{4 \pi^2 } \int_{P} \int_{\mathbb{R}-ib_0} \frac{K(z_1,z_2) \Psi_{++}(z_1,z_2)}{K_{\circ-}(\alpha_1,z_2)(z_2-\alpha_2)(z_1-\alpha_1)} dz_1 dz_2 \\
			& -  \frac{K_{\circ-} K_{-\circ}(\alpha_1,\mathfrak{a}_2)}{ K_{\circ-}(\alpha_1,\mathfrak{a}_2)K_{-\circ}(\mathfrak{a}_1,\mathfrak{a}_2) (\alpha_1 -\mathfrak{a}_1)(\alpha_2-\mathfrak{a}_2)}. \numberthis  \label{eq.PhiCont1st}
		\end{align*}	
		As before, using Hartogs' and Morera's theorems we find that $\Phi$ is analytic in $(H^-(-b_0)\setminus \{\mathfrak{a}_1\}) \times (H^-(\varepsilon) \setminus \{\mathfrak{a}_2\})$. Similarly, using \eqref{eq.PsiCont2}, we find analyticity of $\Phi$ in 
		$(H^-(\varepsilon)\setminus \{\mathfrak{a}_1\}) \times (H^-(-b_0) \setminus \{\mathfrak{a}_2\})$.  As $\Psi_{++}$ and $K$ are analytic in $\S \times \S$ so is $\Phi$ and therefore we find analyticity of $\Phi$ in $\left(H^-(\varepsilon) \setminus \{\mathfrak{a}_1\}\right) \times \left(H^-(\varepsilon) \setminus \{\mathfrak{a}_2\} \right)$. It remains to discuss continuity of $\Phi$ on $P \times P$. Due to the properties of $\sqrt[\rightarrow]{z}$, continuity on this set is clear for all terms in \eqref{eq.PhiCont1st}	except for the integral expression
		\[
		J(\aalpha) = \int_{P} \int_{\mathbb{R}-ib_0}  \frac{K(z_1,z_2) \Psi_{++}(z_1,z_2)}{K_{\circ-}(\alpha_1,z_2)(z_2-\alpha_2)(z_1-\alpha_1)} dz_1 dz_2,
		\]
		where the polar factor $(z_2 - \alpha_2)$ is problematic. But for $\alpha_2$ close to $P$, we can change the contour from $P$ to $P'$ which encloses $\alpha_2$ and $h^-_1 \cup h^-_2$, see Figure \ref{fig:ContourDeform1}. This picks up a residue of the integrand at $z_2=\alpha_2$ which is given by
		\begin{align*}
			\frac{K(z_1,\alpha_2) \Psi_{++}(z_1,\alpha_2)}{K_{\circ-}(\alpha_1,\alpha_2)(z_1-\alpha_1)}.
		\end{align*} 
		This residue has the required continuity as can be seen from \eqref{eq.PsiContFormula2}, so we can safely let $\alpha_2 \to \alpha^{\star}_2$ for any $\alpha^{\star}_2 \in h^-_1 \cup h^-_2$, which gives the sought continuation. 
		\begin{figure}[h]
			\centering
			\includegraphics[width=1\textwidth]{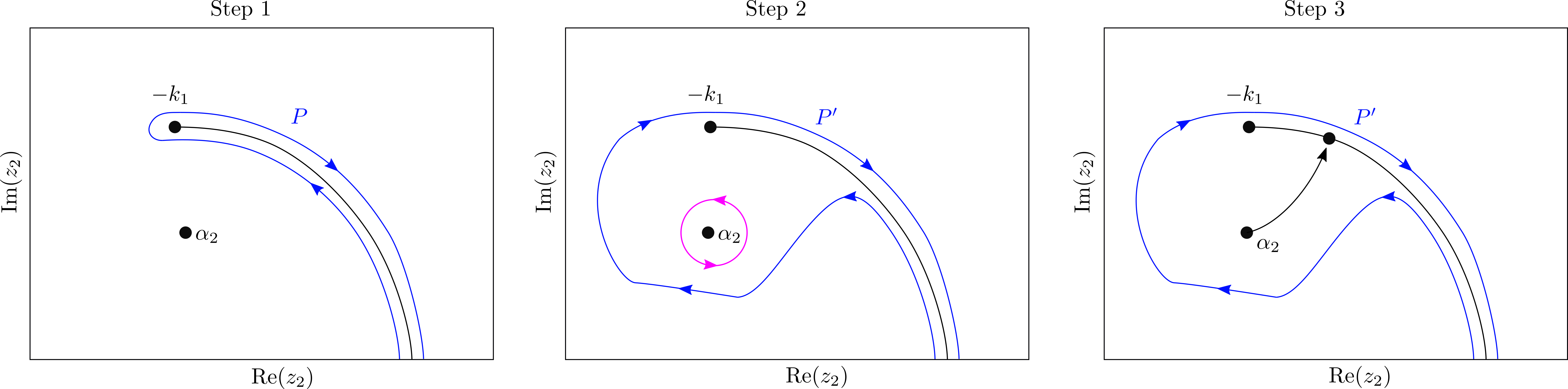}
			\caption{Visualisation of the change of contour performed in Theorem \ref{PhiCont}'s proof. For better visualisation, we only show the change locally about $h^-_1$. After the residue is `picked up' in Step 2, we can safely let $\alpha_2 \to \alpha^{\star}_2 \in h^-_1$ in Step 3 since the singularities of the integrand now lie completely on the contour $P'$ and the additional residue term poses no problems.}
			\label{fig:ContourDeform1}
		\end{figure}
		The residue of $\Phi$ at the pole $\alpha_1=\mathfrak{a}_1$ can be computed explicitly from \eqref{eq.PhiCont1st} since only the external additive term is singular at $\alpha_1=\mathfrak{a}_1$. Similarly, the residue of $\Phi$ at $\alpha_2=\mathfrak{a}_2$ is computed.
		{\hfill{}}
	\end{proof}
	
	The domain of analyticity of $\Phi$ is shown in Figure \ref{fig:Domain3/4} below.
	\begin{figure}[h]
		\centering
		\includegraphics[width=0.625\textwidth]{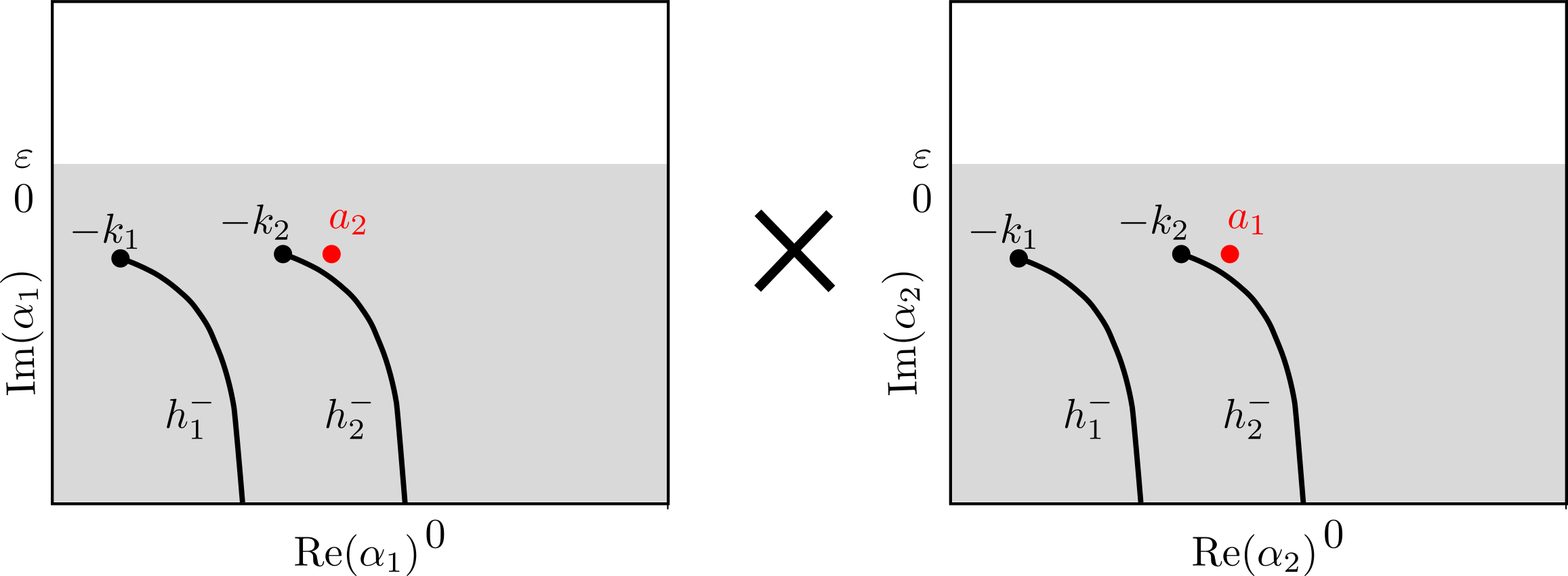}
		\caption{Domain of analyticity of $\Phi = K \Psi_{++}$ and $\Phi_{3/4}$. Polar singularities $\alpha_2 \equiv \mathfrak{a}_2$ and $\alpha_1 \equiv \mathfrak{a}_1$ are shown in red whereas the branch lines $h^-_1$ and $h^-_2$ are shown in black.}
		\label{fig:Domain3/4}
	\end{figure}	
	Recall that, by the Wiener-Hopf equation \eqref{eq.WienerHopf}, we have 
	\[
	\Phi_{3/4} = -K\Psi_{++} - P_{++} = -\Phi - P_{++}
	\]
	so the analyticity properties  of $\Phi$  hold for $\Phi_{3/4}$ as well. That is, $\Phi_{3/4}$ is analytic within the domain shown in Figure \ref{fig:Domain3/4}.
	
	\subsection{Singularities of spectral functions}\label{subsec:Singularities}
	
	Since we ultimately wish to let $\Im(k_1) \to 0$ and $\Im(k_2) \to 0$, let us investigate which singularities we would expect on $\mathbb{R}^2$, the surface of integration  in
	\begin{align}
		\psi(\x) = \frac{1}{4 \pi^2} \int\hspace{-.2cm}\int_{\mathbb{R}^2} \Psi_{++}(\aalpha) e^{-i \x \cdot \aalpha} d \aalpha, \label{eq.psi2} \\
		\phi_{\mathrm{sc}}(\x) = \frac{1}{4 \pi^2} \int\hspace{-.2cm}\int_{\mathbb{R}^2} \Phi_{3/4}(\aalpha) e^{-i \x \cdot \aalpha} d \aalpha. \label{eq.phi2}
	\end{align}
	This set, i.e.\! the intersection of singularities of $\Psi_{++}$ and $\Phi_{3/4}$ with $\mathbb{R}^2$, is henceforth referred to as the `real trace' of the singularities. By \cite{AssierShaninKorolkov2022}, knowledge of the real trace of the singularities is crucial to compute far-field asymptotics of $\phi_{\mathrm{sc}}$ and $\psi$.  	
	According to our previous analysis, assuming that our formulae hold in the limit 	$\Im(k_1) \to 0$ and $\Im(k_2) \to 0$, we find 
	\begin{gather*}
		\Re(\alpha_1) \equiv \mathfrak{a}_1, \ \Re(\alpha_2) \equiv \mathfrak{a}_2; \ \text{ polar sets, } \numberthis \label{eq.sing1} \\
		\Re(\alpha_1)  \equiv -k_1, \ \Re(\alpha_1)  \equiv -k_2, \ \Re(\alpha_2) \equiv -k_1, \ \Re(\alpha_2) \equiv -k_2; \ \text{ branch sets. } \numberthis \label{eq.sing2}
	\end{gather*}
	However, note that as $\Im(k_2) \to 0$ we also expect parts of the (complexified) circle defined by 
	$\alpha^2_1 + \alpha^2_2 = k^2_2$ to become singular points of $\Psi_{++}$: From the analytical continuation procedure, we know that such singularities can only `come from'  $H^- \setminus \{\mathfrak{a}_1\} \times H^- \setminus \{\mathfrak{a}_2\} $. However, we know how, exactly, $\Psi_{++}$ can be represented in $H^- \setminus \{\mathfrak{a}_1\} \times H^- \setminus \{\mathfrak{a}_2\}$, namely by \eqref{eq.PsiCont}. Therefore, to unveil $\Psi_{++}$'s singularities in $H^- \setminus \{\mathfrak{a}_1\} \times H^- \setminus \{\mathfrak{a}_2\}$, we just need to analyse the external term in \eqref{eq.PsiCont}, since the integral term is by construction analytic in $H^- \times H^-$. Now, in $H^- \setminus \{\mathfrak{a}_1\} \times H^- \setminus \{\mathfrak{a}_2\}$, the external factor is only singular for 
	\begin{align}
		K_{\circ +} = \frac{\sqrt[\rightarrow]{k^2_2 - \alpha^2_1} + \alpha_2}{\sqrt[\rightarrow]{k^2_1 - \alpha^2_1} + \alpha_2} =0 
	\end{align}
	i.e.\! whenever 
	\begin{align}
		\sqrt[\rightarrow]{k^2_2 - \alpha_1^2} = - \alpha_2, \label{eq.Sing1}
	\end{align}
	since by definition of $H^- \setminus \{ \mathfrak{a}_j\}, \ j=1,2$, the singularity sets given in \eqref{eq.sing1} and \eqref{eq.sing2} do not belong to $H^- \setminus \{\mathfrak{a}_j\}, \ j=1,2$. 
	As the branch of the square root is chosen such that $\sqrt[\rightarrow]{k^2_2} = k_2$, we find that if $\Re(\alpha_1) =0$, we must have $\Re(\alpha_2) =- k_2$, giving the first real singular point of  \eqref{eq.Sing1}. Now, by continuity, we find that $\Re(\alpha_2) \leq 0$ for all $\Re(\alpha_2)$ satisfying \eqref{eq.Sing1}. However, $\Re(\alpha_1)$ can take all values between $-k_2$ and $k_2$. 
	
	Similarly, from \eqref{eq.PsiCont2}
	we find that $\Psi_{++}$ is singular in $H^- \setminus \{\mathfrak{a}_1\} \times H^- \setminus \{\mathfrak{a}_2\}$ when 
	\begin{align}
		K_{+ \circ} = \frac{\sqrt[\rightarrow]{k^2_2 - \alpha^2_2} + \alpha_1}{\sqrt[\rightarrow]{k^2_1 - \alpha^2_2} + \alpha_1} =0
	\end{align}
	i.e.\! whenever
	\begin{align}
		\sqrt[\rightarrow]{k^2_2 - \alpha_2^2} = - \alpha_1. \label{eq.Sing2}
	\end{align}
	Then, just as before we find that \eqref{eq.Sing2} is satisfied for all $\Re(\alpha_2) \in [-k_2, k_2]$ and $\Re(\alpha_1) \leq 0$.

	Therefore, the real trace of the complexified circle $\{\aalpha \in \mathbb{C}^2| \ \alpha^2_1 + \alpha^2_2 = k^2_2\}$ that is a singularity of $\Psi_{++}$ can only be the intersection of the sets of solutions to \eqref{eq.Sing1} and \eqref{eq.Sing2}, i.e.\! the set  
	\begin{align*}
		\left\{\aalpha \in \mathbb{R}^2| \ \alpha_1^2 + \alpha^2_2 = k^2_2, \ \alpha_1 \leq 0, \ \alpha_2 \leq 0  \right\}.
	\end{align*}
	
	Similarly, using $\Phi = K \Psi_{++}$ to analyse the behaviour of $\Phi$ in $$\left(\UHP \times \mathbb{C} \setminus \left(h^-_1 \cup h^-_2 \cup \{\mathfrak{a}_1\}\right) \right) \cup \left( \mathbb{C} \setminus \left(h^-_1 \cup h^-_2 \cup \{\mathfrak{a}_2\}\right) \times \UHP\right),$$ we find that the part of the circle's real trace on which we expect $\Phi_{3/4}$ to be singular is given by
	\begin{align*}
		\left\{\aalpha \in \mathbb{R}^2| \ \alpha_1^2 + \alpha^2_2 = k^2_1, \ \alpha_1 \geq 0, \text{ or } \ \alpha_2 \leq 0  \right\}.	
	\end{align*}
	The real traces of the singularities are shown in Figure \ref{fig:RealTracesSingularities}	below. \\
	
	\begin{figure}[h]
		\includegraphics[width=\textwidth]{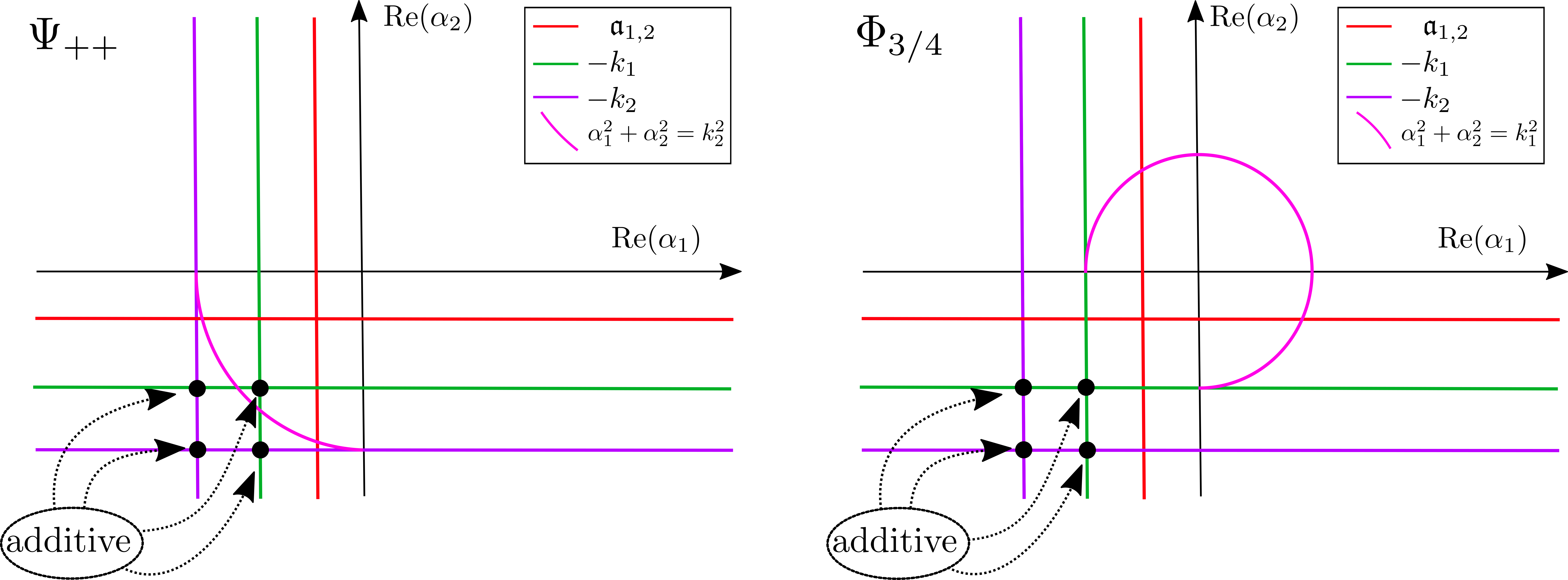}
		\caption{Real trace of the singularities of $\Psi_{++}$ and $\Phi_{3/4}$  in the case $\vartheta_0 \in (\pi, 3 \pi/2)$. The `additive' crossing of branch sets refers to the additive crossing property discussed in Section \ref{sec:AdditiveCrossing}. 
		}
		\label{fig:RealTracesSingularities}
	\end{figure}

	\red{\noindent \textbf{Change of incident angle.}  Let us now consider the case $\vartheta_0 \in (\pi/2, \pi)$. Due to symmetry, the case $\vartheta_0 \in(3 \pi/2, 2 \pi)$ can be dealt with similarly.  We  now treat $\vartheta_0$ as a parameter within the formulae for analytic continuation \eqref{eq.PsiContFormula1}, \eqref{eq.PsiContFormula2}, \eqref{eq.PsiCont}, and \eqref{eq.PsiCont2} of $\Psi_{++}$. This yields formulae for $\Psi_{++}$ when $\vartheta_0 \in (\pi/2, \pi)$. We then obtain new singularities within these formulae for analytic continuation. Namely, the external additive term in \eqref{eq.PsiContFormula1} becomes singular at $\{\aalpha_1 \equiv  - \sqrt[\rightarrow]{k_1^2 - \a_2^2}\}$. This procedure therefore yields a new singularity of $\Psi_{++}$ and $\Phi_{3/4}$ within $\LHP \times \mathbb{C}$. The real traces of the spectral functions' singularities in this case are shown in Figure \ref{fig:RealTracesCompl01}. Note that we may not allow $\vartheta_0 \in (0, \pi/2)$. This is because such change of incident angle changes the incident wave's wavenumber from $k_1$ to $k_2$, and therefore such change cannot be assumed to be continuous.
			\begin{center}
				\begin{figure}[h!]
					\includegraphics[width=\textwidth]{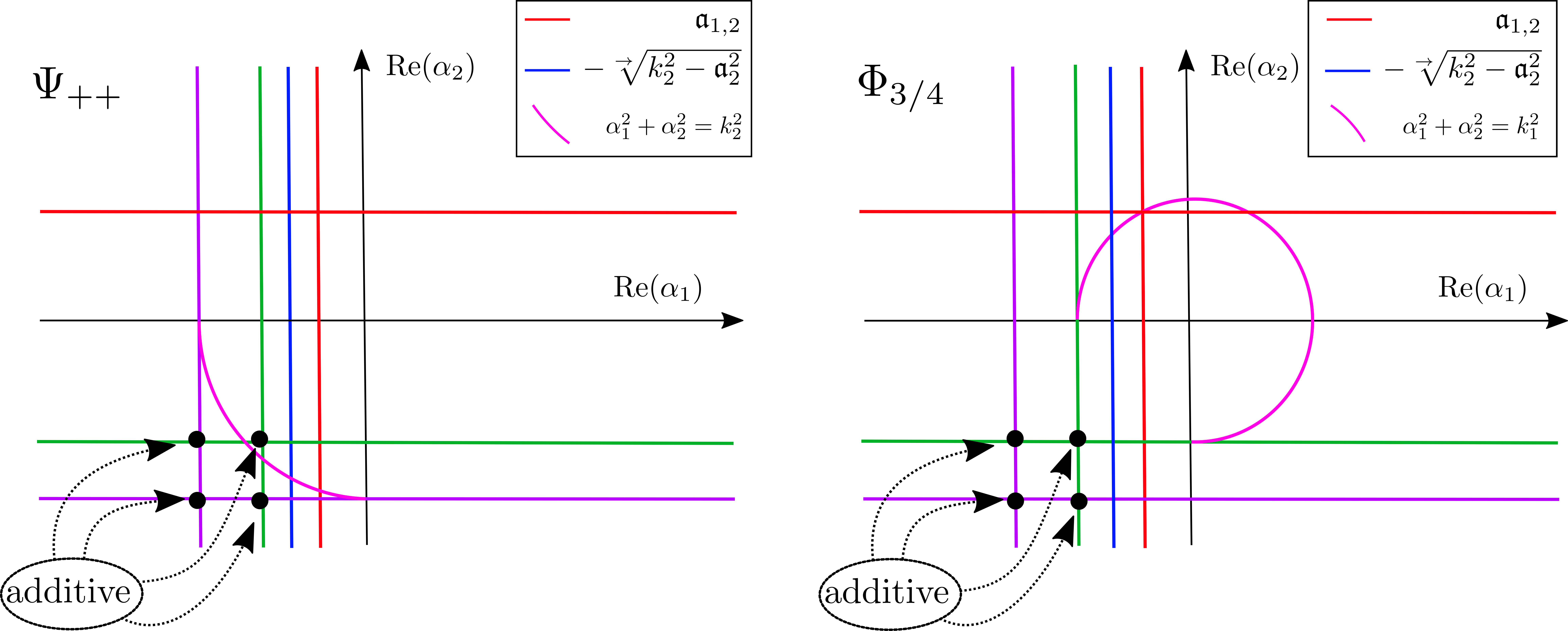}
					\caption{\red{Real traces of $\Psi_{++}$ and $\Phi_{3/4}$'s singularities in the case $\vartheta_0 \in (\pi/2, \pi)$. \BLUE{Again, the branch sets at $\alpha_{1,2} \equiv -k_1$ and $\alpha_{1,2} \equiv -k_2$ are coloured in green and purple, respectively.} Comparing with Figure \ref{fig:RealTracesSingularities}, we see the newly apparent singularity at $\alpha_1 \equiv - \sqrt[\rightarrow]{\blue{k_2}^2 - \a_2^2}$ (coloured in \BLUE{blue}) and that the singularity at $\alpha_2 \equiv \a_2$ has changed half-plane.}}
					\label{fig:RealTracesCompl01}
				\end{figure}
			\end{center}
		
	\begin{remark}[Failure of limiting absorption principle] In the case of $\vartheta_0 \in (\pi/2, \pi)$, we cannot directly impose the radiation condition \magenta{on the scattered and transmitted fields} via the limiting absorption principle, although, of course, a radiation condition still needs to be imposed. The failure of defining the radiation condition via the absorption principle is due to the fact that for positive imaginary part $\varkappa >0$ of $k_1$ and $k_2$, such incident angle changes the sign of $\a_2$: Whereas for $\vartheta_0 \in (\pi, 3 \pi/2)$ we are guaranteed $\Im(\a_1), \Im(\a_2) < 0$ whenever $\varkappa >0$ we now have $\Im(\a_1) <0,$ and $\Im(\a_2) > 0$ whenever $\varkappa >0$. Thus, when $\vartheta_0 \in (\pi/2, \pi)$, one has to carefully choose the `indentation' of $\mathbb{R}^2$ around the real traces of the singularities such that the radiation condition remains valid. Here, `indentation' refers to the novel concept of `bridge and arrow configuration' which is extensively discussed in \cite{AssierShaninKorolkov2022}. We plan to address this difficulty in future work.
	\end{remark}	
		}		
	\section{The additive crossing property}\label{sec:AdditiveCrossing}
	
	We want to investigate the behaviour of $\Phi_{3/4}$ on $P \times P$. In particular, we wish to investigate whether the additive crossing property introduced in \cite{AssierShanin2019} is satisfied with respect to the points $(-k_1,-k_1), (-k_2,-k_2), (-k_1,-k_2),$ and $(-k_2,-k_1)$ which are the points at which the branch sets are `crossing', see Figure \ref{fig:RealTracesSingularities}. Other than yielding a criterion for 3/4-basedness in the quarter-plane problem (cf.\! \cite{AssierShanin2019}), this property was also crucial to solving the simplified quarter-plane functional problem corresponding to having a source located at the quarter-plane's tip, see \cite{AssierShanin2021VertexGreensFunctions}. It also emerged in the different context of analytical continuation of real wave-fields defined on a Sommerfeld surface, see \cite{AssierShanin2021AnalyticalCont}. Therefore, it seems that the property of additive crossing is strongly related to the physical behaviour of the corresponding wave-fields. \red{\BLUE{Indeed, the additive crossing property is crucial to obtaining the correct far-field asymptotics as it prohibits the existence of unphysical waves, \Blue{as shown in} \cite{AssierShaninKorolkov2022}.}}

	We begin by studying 
	\begin{align}
		\phi_{\mathrm{sc}}(\x) = \frac{1}{4 \pi^2} \int\hspace{-.2cm}\int_{\mathbb{R}^2} \Phi_{3/4}(\boldsymbol{\alpha})e^{-i\boldsymbol{\alpha} \cdot \x} d\boldsymbol{\alpha}. \label{eq.PhiFT}
	\end{align}
	Let us investigate what happens when we change the domain of integration from $\mathbb{R} \times \mathbb{R}$ to $P \times P$ in \eqref{eq.PhiFT}. Due to the asymptotic behaviour of $\Phi_{3/4}$ (cf.\! Section \ref{thm:AsymptoticSpectral}), we will not obtain any `boundary terms at infinity'. 
		
		However, we have to account for the polar sets $\alpha_1 \equiv \mathfrak{a}_1$ and $\alpha_2 \equiv \mathfrak{a}_2$. As in \cite{AssierShanin2019}, this change of contour yields
		
		\begin{align*}
			4 \pi^2 \phi_{\mathrm{sc}}(\x) =&  \int_{P}\int_{P} \Phi_{3/4}(\boldsymbol{\alpha}) e^{-i\boldsymbol{\alpha}\cdot \x} d\boldsymbol{\alpha}\\
			& - 2 \pi i \int_{P}\res{\alpha_1=\mathfrak{a}_1}\left(\Phi_{3/4}(\boldsymbol{\alpha}) e^{-i\boldsymbol{\alpha}\cdot \x}\right) d\alpha_2 
			- 2 \pi i\int_{P} \res{\alpha_2 = \mathfrak{a}_2} \left(\Phi_{3/4}(\boldsymbol{\alpha}) e^{-i\boldsymbol{\alpha}\cdot \x}\right) d\alpha_1 \\
			& + 4 \pi^2
			\res{\alpha_2=\mathfrak{a}_2}\left(\res{\alpha_1=\mathfrak{a}_1} \Phi_{3/4}(\boldsymbol{\alpha}) e^{-i\boldsymbol{\alpha}\cdot \x}\right). \numberthis \label{eq.phiLongInt}
		\end{align*}
		But using \eqref{eq.PhiRes1}--\eqref{eq.PhiRes2} and the fact that $\Phi_{3/4}=-\Phi -P_{++}$
		we find that 
		\begin{align*}
			\res{\alpha_1=\mathfrak{a}_1}\left(\Phi_{3/4}(\boldsymbol{\alpha}) e^{-i\boldsymbol{\alpha}\cdot \x}\right) \text{ and }
			\res{\alpha_2 = \mathfrak{a}_2} \left(\Phi_{3/4}(\boldsymbol{\alpha}) e^{-i\boldsymbol{\alpha}\cdot \x}\right)
		\end{align*}
		are continuous at $h^-_1 \cup h^-_2$, so their integral over $P$ vanishes. Moreover, using \eqref{eq.PhiRes1} and $\Phi_{3/4}= -\Phi - P_{++}$, we find that the double residue in \eqref{eq.phiLongInt} vanishes. Therefore, we find:
		
		\begin{lemma} \
			The scattered field $\phi_{sc}$ satisfies 	
			\begin{align}
				\phi_{sc}(\x) = \frac{1}{4 \pi} \int_{P}\int_{P} \Phi_{3/4}(\boldsymbol{\alpha})e^{-i \boldsymbol{\alpha} \cdot \x} d\boldsymbol{\alpha}.
			\end{align}
			In particular, since, by Theorem \ref{thm:BackToPhysical}, $\phi_{\mathrm{sc}}(\x) =0$ in $Q_1$, we have
			\begin{align}
				\int_{P}\int_{P} \Phi_{3/4}(\boldsymbol{\alpha})e^{-i \boldsymbol{\alpha} \cdot \x} d\boldsymbol{\alpha} = 0, \ \forall \x \in Q_1. \label{eq.DoublePZeroes}
			\end{align}
		\end{lemma}	
		
		\subsection{Three quarter-basedness and additive crossing}
	
		Recall that  $\Phi_{3/4}$ is defined on $P$ by continuity. We can therefore define values of $\Phi_{3/4}$ on $h^-_{1,2}$ depending on whether $\alpha \in h^-_{1,2} \subset \mathbb{C}$ is approached from the left, or the right, see Figure \ref{fig:LeftRightShore}.
		\begin{figure}[h!]
			\centering
			\includegraphics[width=.45\textwidth]{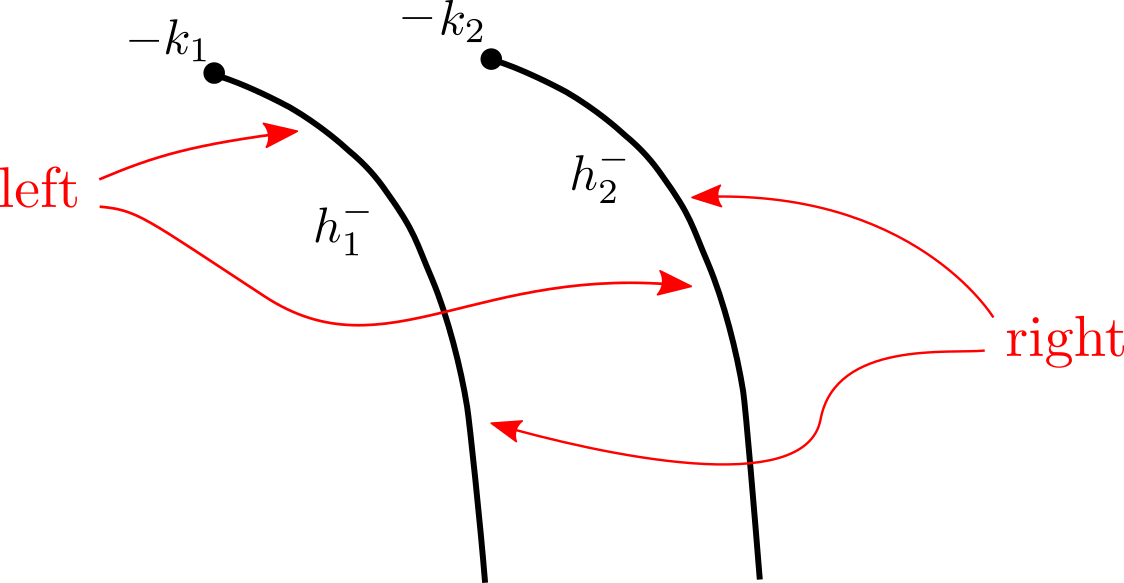}
			\caption{Visualisation of `left and right side' of the cuts $h^-_{1,2}$. }
			\label{fig:LeftRightShore}
		\end{figure}
		Let  $\alpha^{l,r}_{1,2}$ denote the values on the left (resp. right) side of $h^-_1$ and $h^-_2$, respectively, and define {\fontsize{8pt}{0pt}\selectfont\begin{align*}
				\Phi_{AC}(\aalpha) = \begin{cases}
				\Phi_{3/4}(\alpha^r_1,\alpha^r_2) + \Phi_{3/4}(\alpha^l_1,\alpha^l_2)  -\Phi_{3/4}(\alpha^l_1,\alpha^r_2) 
				-\Phi_{3/4}(\alpha^r_1,\alpha^l_2), \text{ if } \aalpha \in (h^-_1 \times h^-_1) \cup (h^-_2 \times h^-_2), \\[.1em]
				\Phi_{3/4}(\alpha^l_1,\alpha^r_2) 
				+\Phi_{3/4}(\alpha^r_1,\alpha^l_2)  -\Phi_{3/4}(\alpha^r_1,\alpha^r_2) - \Phi_{3/4}(\alpha^l_1,\alpha^l_2),  \text{ if }  \aalpha \in (h^-_2 \times h^-_1) \cup (h^-_1 \times h^-_2).
			\end{cases}
		\end{align*}}
	 
	 \noindent \RED{Using the analyticity properties as well as the asymptotic behaviour \eqref{Psi++Behaviour1}--\eqref{Psi++Behaviour2} of $\Phi_{3/4}$, it can be shown that $\Phi_{AC}$ is Lipschitz continuous on $h^-_1 \cup h^-_2$.}
		We now rewrite \eqref{eq.DoublePZeroes} as
		\begin{align}
			\int_{P} \int_{P} \Phi_{3/4}(\boldsymbol{\alpha}) e^{-i \boldsymbol{\alpha} \cdot \x} d\boldsymbol{\alpha} =
			\int_{h^-_1\cup h^-_2} \int_{h^-_1\cup h^-_2}	\Phi_{AC}(\boldsymbol{\alpha}) e^{-i\boldsymbol{\alpha}\cdot \x} d\boldsymbol{\alpha} =0, \ \forall \x \in Q_1. \label{eq.ACInt}
		\end{align}
		The equality \eqref{eq.ACInt} is always satisfied if $\Phi_{AC} =0$ on $(h^-_1\cup h^-_2) \times (h^-_1\cup h^-_2)$. But
		in fact, since $\Phi_{AC}$ is Lipschitz continuous, we can apply Corollary \ref{thm.LaplaceLike2} and we find that \eqref{eq.ACInt} is equivalent to 
		\begin{align*}
			\	\Phi_{AC}(\boldsymbol{\alpha}) = 0, \ \forall \ \boldsymbol{\alpha} \in (h^-_1 \cup h^-_2) \times  (h^-_1\cup h^-_2), \numberthis \label{eq.ACProperty}
		\end{align*}
		i.e.\! \eqref{eq.ACInt} is equivalent to the fact that $\Phi_{3/4}$ \red{\BLUE{satisfies}}
		\begin{align}
			\Phi_{3/4}(\alpha^r_1,\alpha^r_2) + \Phi_{3/4}(\alpha^l_1,\alpha^l_2) -\Phi_{3/4}(\alpha^l_1,\alpha^r_2) 
			-\Phi_{3/4}(\alpha^r_1,\alpha^l_2) = 0 \label{eq.ACProperty3}
		\end{align}	
		for all $(\alpha_1,\alpha_2) \in (h^-_1 \cup h^-_2) \times  (h^-_1\cup h^-_2)$. Note that the Corollary can be applied due to the minus in front of the exponential in \eqref{eq.ACInt}, so in the setting of Corollary \ref{thm.LaplaceLike2} we choose the lines $\{x_1 \leq 0, x_2 =0\}$ and $\{x_1=0, x_2 \leq 0 \}$. Although \eqref{eq.ACProperty} and \eqref{eq.ACProperty3} depend on the choice of branch cuts $h^-_1$ and $h^-_2$, it can be shown that if \eqref{eq.ACProperty3} holds for one choice of branch cuts, it holds for every choice. Therefore, it makes sense to say that the equality \eqref{eq.ACProperty3} holds with respect to the points $(-k_j,-k_l), \ k,l =1,2$ i.e.\! with respect to the points of crossing of branch sets, as illustrated in Figure \ref{fig:RealTracesSingularities}.  
		\red{\BLUE{Moreover, since $\Phi_{3/4}$ is bounded near its branch sets,  \eqref{eq.ACProperty3} is sufficient to prove that $\Phi_{3/4}$ satisfies the following  \emph{additive crossing property}:  
		There exists some neighbourhood $U_{j,l} \subset \mathbb{C}^2$ of $(k_{j},k_l), \ j,l=1,2$, such that
			\begin{align}
				\Phi_{3/4}(\aalpha) = F_{1j}(\aalpha) + F_{2l}(\aalpha), \forall \aalpha \in U_{j,l}, \ j,l=1,2, \label{eq.ACProperty2}
			\end{align}	
		where $F_{1j}(\aalpha)$ is regular at $\alpha_1 \equiv - k_j$, and $F_{2l}$ is regular at $\alpha_2 \equiv -k_l$. The proof is identical to the corresponding proof \Blue{of the additive crossing property satisfied by} the spectral function \Blue{of} the quarter-plane problem (see \cite{AssierShanin2019} Section 4), and hence omitted. As mentioned at the beginning of Section \ref{sec:AdditiveCrossing}, the additive crossing property is directly linked to the far-field behaviour of the scattered and transmitted fields.
		}}

		\subsection{Reformulation of the functional problem}\label{sec:FunctionalProblemReformulated}
		
		\Blue{Finally, we obtain}
		the following reformulation of Theorem \ref{thm:BackToPhysical}.
		
		\begin{theorem}\label{thm:secondfunctional}
			Let $P_{++}$ and $K$ be as in \eqref{eq.PandKDefinition}. Let $\Psi_{++}$ satisfy the following properties:
			\begin{enumerate}
				\item \ $\Psi_{++}$ is analytic in 	\[\left(\UHP \times \mathbb{C} \setminus \left(h^-_1 \cup h^-_2 \cup \{\mathfrak{a}_1\}\right) \right) \cup \left( \mathbb{C} \setminus \left(h^-_1 \cup h^-_2 \cup \{\mathfrak{a}_2\}\right) \times \UHP\right),\]
				\item \ There exists an $\varepsilon>0$ such that the function $\Phi_{3/4}$ defined by $\Phi_{3/4} = -K \Psi_{++} - P_{++}$ is analytic in 
				\[
				\left(H^-(\varepsilon)\setminus \{\mathfrak{a}_1\}\right) \times \left(H^-(\varepsilon)\setminus \{\mathfrak{a}_2\}\right)
				\]
				with simple poles at $ \alpha_1 =\mathfrak{a}_1$ and $\alpha_2=\mathfrak{a}_2$,
				\item \ The residues of $-\Phi_{3/4} - P_{++}$ at the poles $ \alpha_1 =\mathfrak{a}_1$ and $\alpha_2=\mathfrak{a}_2$ are given by \eqref{eq.PhiRes1} and \eqref{eq.PhiRes2},
				\item \ $\Phi_{3/4}$ is continuous on $P \times P$ and satisfies the additive crossing property for each of the following points: $(-k_1,-k_1), \ (-k_1,-k_2), \ (-k_2,-k_2),$ and $(-k_2,-k_2)$,
				\item \ The functions $\Psi_{++}$ and $\Phi_{3/4}$ have the asymptotic behaviour \eqref{Psi++Behaviour1}--\eqref{Psi++Behaviour2}.
			\end{enumerate}	
			Then the fields $\phi_{\mathrm{sc}}$ and $\psi$ defined by
			\begin{align}
				\phi_{sc}(\x) = \frac{1}{4 \pi^2} \int\hspace{-.2cm}\int_{\mathbb{R}^2} \Phi_{3/4}(\aalpha)e^{-i \aalpha \cdot \x} d\aalpha \  \text{ and } \ 
				\psi(\x) = \frac{1}{4 \pi^2} \int\hspace{-.2cm}\int_{\mathbb{R}^2} \Psi_{++}(\aalpha)e^{-i \aalpha \cdot \x} d\aalpha  \label{eq.physicalfields}
			\end{align} 
			satisfy the penetrable wedge problem defined by \eqref{eq:1.1}--\eqref{eq:1.6} with respect to the incident wave $\phi_{\iin}(\x) = \exp(-i(\mathfrak{a}_1x_1+\mathfrak{a}_2x_2))$. Moreover, they satisfy the Meixner conditions \eqref{eq.2.56}--\eqref{eq.2.57} as well as the Sommerfeld radiation condition.
		\end{theorem}
		
		The theorem's proof is immediate since, according to Sections \ref{sec:AnalyticalContinuation}--\ref{sec:AdditiveCrossing}, the conditions 1--5 of Theorem \ref{thm:secondfunctional} imply that $\Psi_{++}$ and $\Phi_{3/4}$ satisfy the penetrable wedge functional problem and therefore Theorem \ref{thm:BackToPhysical} holds for $\phi_{\mathrm{sc}}$ and $\psi$.
		
		\section{Concluding remarks}\label{sec:ConcludingRemarks}
		
		We have shown that the novel additive crossing property\Blue{, which was} introduced in \cite{AssierShanin2019} in the context of diffraction by a quarter-plane\Blue{,} holds \Blue{for the problem} of diffraction by a penetrable wedge. Indeed, in similarity to the one dimensional Wiener-Hopf technique, the spectral functions' \Blue{singularities within $\mathbb{C}^2$} solely depend on the kernel $K$ and the forcing $P_{++}$, and therefore the techniques developed by Assier and Shanin in \cite{AssierShanin2019} could be adapted to the penetrable wedge diffraction problem, once equivalence of the penetrable wedge functional problem and the physical problem was shown in Section \ref{subsec:Reform1st}. 
		However, as in \cite{AssierShanin2019}, we cannot apply Liouville's theorem since, according to Theorem \ref{thm:secondfunctional}, the domains of analyticity of the unknowns $\Psi_{++}$ and $\Phi_{3/4}$ span all of $\mathbb{C}^2$ \emph{minus some set of singularities}. 
		
		Nonetheless, using the in Section \ref{subsec:Singularities} established real traces of the spectral functions' singularities, we expect to be able to obtain far-field asymptotics of the physical fields using the framework developed in \cite{AssierShaninKorolkov2022}. In particular, as in \cite{AssierAbrahams2021}, we expect the diffraction coefficient in $\mathbb{R}^2 \setminus \text{PW}$ (resp. $\text{PW}$) to be proportional to $\Psi_{++}$ (resp. $\Phi$) evaluated at a given point. \red{Moreover, we expect that a similar phenomenon holds for the lateral waves. That is, we expect that the results of the present article and \cite{AssierShaninKorolkov2022} allow us to represent the lateral waves such that their decay and phase are explicitly known, whereas their coefficients are proportional to, say, $\Psi_{++}$, evaluated at a given point.} Thus, we expect to be able to use the results of \cite{Kunz2021diffraction} to accurately approximate the far-field in the spirit of \cite{AssierAbrahams2020}. We moreover plan to test \Blue{far-field} accuracy of Radlow's erroneous ansatz (\Blue{which was given in} \cite{Radlow1964PW}).
		
		Finally, we note that Liouville's theorem is not only applicable to functions in $\mathbb{C}^2$ but also to functions defined on suitably `nice' complex manifolds, see \cite{LiZhangZhang2018,Lin1988}. Therefore, gaining a better understanding of the complex manifold on which $\Psi_{++}$ and $\Phi_{3/4}$ are defined 
		could be crucial to completing the 2D Wiener-Hopf technique. Note that this final step is presumably easier for the penetrable wedge than for the quarter-plane since in the latter, the real trace of the complexified circle is a branch set (see \cite{AssierShanin2019}) which could drastically change the topology of the sought complex manifold. \\

\noindent \textbf{Funding:} The authors would like to acknowledge funding by EPSRC (EP/W018381/1 and EP/N013719/1) for RCA and a University of Manchester Dean’s scholarship award for VDK.
		
\appendix
		
		
	\section{Uniqueness theorems}\label{app.Uniqueness}
		
		\begin{theorem}\label{lem:FromIntegralToInterface}
			
			For $j=1,2,3,4$, let $f_{j}: [0, \infty) \to \mathbb{C}$ be integrable and such that $f_j(x) = \mathcal{O}(x^{\nu})$ for some $\nu > -1$ as $x \to 0$. Assume that for all $\alpha_{1}, \alpha_{2} \in \mathbb{R}$, we have 
			\begin{align}
				\int_{0}^{\infty}f_1(x)e^{i\alpha_2x} dx + \int_{0}^{\infty}f_2(x)e^{i\alpha_1x}dx + \alpha_1\int_{0}^{\infty}f_3(x)e^{i\alpha_2x}dx + \alpha_2 \int_{0}^{\infty}f_4(x)e^{i\alpha_1x}dx =0. \label{eq.Unique1}
			\end{align}
			Then $f_1=f_2=f_3=f_4\equiv 0$.
		\end{theorem}
		
		\begin{proof}
			For simplicity, let us set 	
			\[
			F_j(\alpha) = \int_{0}^{\infty}f_j(x) e^{i \alpha x} dx, \text{ for } j=1,2,3,4.
			\]
			Therefore, \eqref{eq.Unique1} can be rewritten as
			\begin{align}
				F_1(\alpha_2) +F_2(\alpha_1) + \alpha_1 F_3(\alpha_2) + \alpha_2 F_4(\alpha_1) =0. \label{eq.F+First}
			\end{align}	
			Now, since $f_j(x) = \mathcal{O}(x^{\nu})$, by the Abelian theorem (cf.\! \cite{Doetsch1974}) we find that for every $j=1,2,3,4$ $F_j(\alpha) = \mathcal{O}(1/\alpha^{\nu +1})$ as $|\alpha| \to \infty$ in $\UHP(0)$. In particular, it implies $F_j(\alpha) \to 0$ as $|\alpha| \to \infty$  in $\UHP(0)$. \\
			
			\textbf{Step 1.} Let $\alpha_2 \equiv 0$ in \eqref{eq.F+First} to obtain
			\begin{align}
				F_1(0) +F_2(\alpha_1) + \alpha_1 F_3(0) =0.
			\end{align}
			Since $F_2(\alpha_1) \to 0$ as $|\alpha_1| \to \infty$, we find $F_1(0) = F_3(0) =0$ and therefore  $F_2(\alpha_1) \equiv 0$. \\
			
			\textbf{Step 2.} Let $\alpha_1 \equiv 0$ in \eqref{eq.F+First} and use the result of the first step to obtain
			\begin{align}
				F_1(\alpha_2)+ \alpha_2F_4(0) =0.
			\end{align}	
			Again, since $F_1(\alpha_2) \to 0$ as $|\alpha_2| \to \infty$ we find $F_4(0) =0$ and therefore $F_1(\alpha_2) \equiv 0$. \\
			
			\textbf{Step 3.} Eq. \eqref{eq.F+First} now becomes
			\begin{align}
				\alpha_1F_3(\alpha_2) + \alpha_2F_4(\alpha_1) =0.
			\end{align}
			Fix $\alpha_1 \equiv \alpha^{\star}_1 \neq 0$. Thus 
			\begin{align}
				F_3(\alpha_2) + \alpha_2 \frac{F_4(\alpha^{\star}_1)}{\alpha^{\star}_1} =0.
			\end{align}
			As before, since $F_3 \to 0$ as $|\alpha_2| \to \infty$, we find $F_4(\alpha^{\star}_1) =0$ and therefore $F_3(\alpha_2) \equiv 0$. Similarly, we find $F_4(\alpha_1) \equiv 0$. By inverse Laplace transform, we find $f_1=f_2=f_3=f_4 \equiv 0$.
			{\hfill{}}
		\end{proof}

		The following is a direct generalisation of \cite{AssierShanin2019} Theorem C.1 (and the techniques used for its proof are almost identical).
		
		\begin{theorem}[1D Uniqueness Theorem]\label{thm.LaplaceLike}
			Let $\gamma_1:[0,\infty) \to \mathbb{C}$ and $\gamma_2: [0,\infty) \to \mathbb{C}$ be piecewise smooth non-(self)intersecting curves lying completely in the sector $\varphi_2<\arg(z)<\varphi_1$, where $\varphi_1-\varphi_2<\pi$, such that $|\gamma_1(t)|, |\gamma_2(t)| \to \infty$ as $t \to \infty$, see Figure \ref{fig:RiemannSphere2} top left. Let $\gamma_1$ (resp. $\gamma_2$) be `finite', in the sense that the length of their segments within the disk $\{ z \in \mathbb{C}| \ 0 \leq |z| \leq r\}$ is finite for all $0 < r <\infty$. Let $f:\gamma_1\cup \gamma_2 \to \mathbb{C}$ satisfy $f(z)=\mathcal{O}(1/|z|^{\beta})$, $\beta>0$, as $|z| \to \infty$ on $\gamma_1$ (resp. $\gamma_2$) and let $f$ be Lipschitz continuous along $\gamma_1 \cup \gamma_2$.
			If there exists a line $L \subset \mathbb{C}$ of constant argument `$\arg(s)$'  such that $-\varphi_2<\arg(s)<\pi-\varphi_1$ (cf.\! Figure \ref{fig:RiemannSphere2} bottom right), and if  
			\begin{align}
				\int_{\gamma_1 \cup \gamma_2} f(z)e^{isz}dz \equiv 0, \ \forall s \in L, \label{eq.AppenedixIntegralUnique}
			\end{align} 
			then $f \equiv 0$ on $\gamma_1 \cup \gamma_2$.
		\end{theorem}
		
		\begin{proof}
			\begin{figure}[h!]
				\centering
				\includegraphics[width=.75\textwidth]{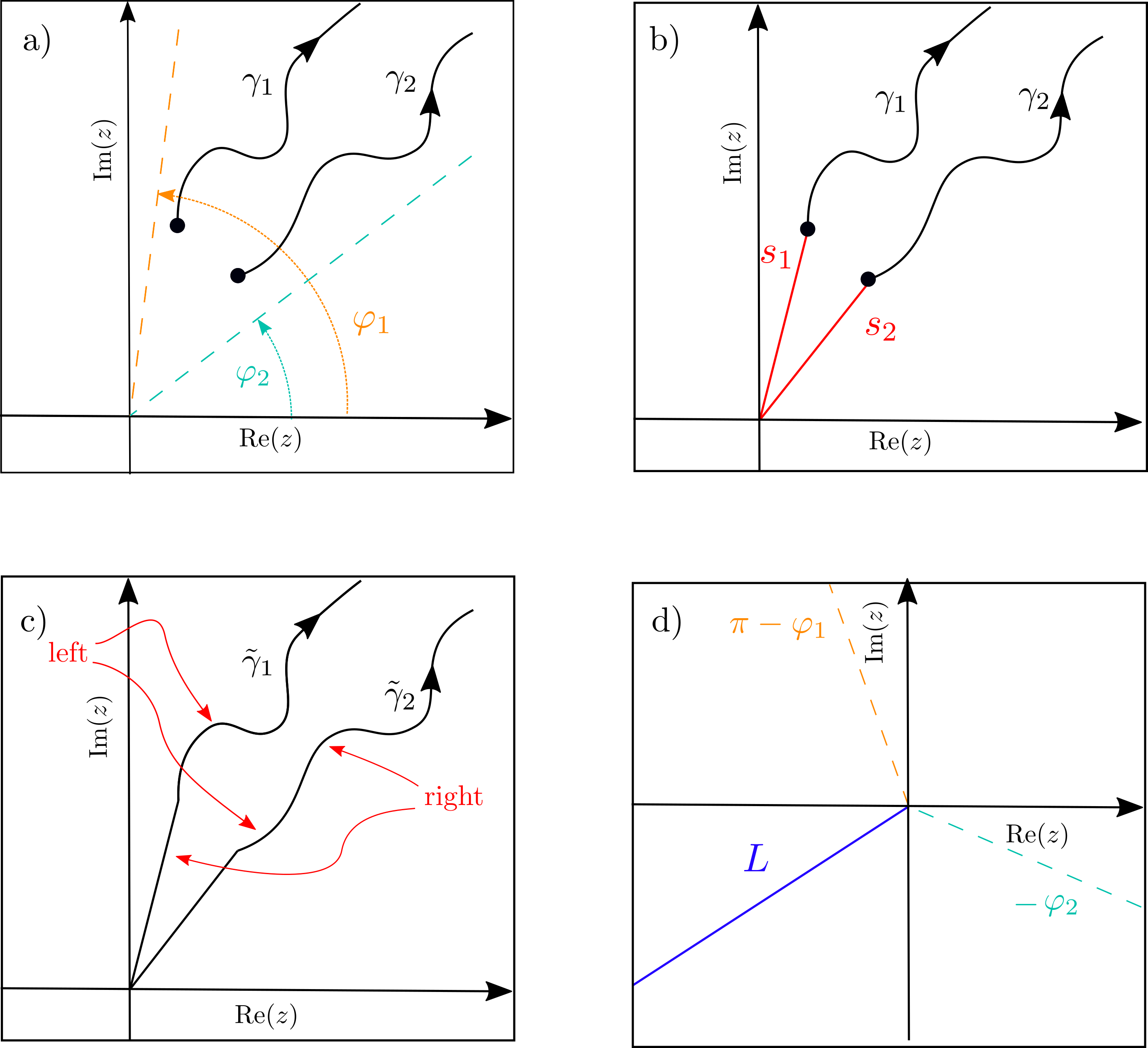}
				\caption{ a) on the top left, shows the curves $\gamma_1$ and $\gamma_2$, and  b), on the top right, shows how the curves are connected with the origin; c) shows how the left and right side of the curves are defined, and d) shows the line $L$ on which \eqref{eq.AppenedixIntegralUnique} is satisfied.} 
				\label{fig:RiemannSphere2}
			\end{figure}
			
			Consider the functions 
			\begin{align*}
				&	y_1(z) = \frac{1}{2 \pi i} \int_{\gamma_1} \frac{f(z')}{z' -z} dz', \\
				&	y_2(z) = \frac{1}{2 \pi i} \int_{\gamma_2} \frac{f(z')}{z' -z} dz'.
			\end{align*}
			We know that $y_1$ (resp. $y_2$) is analytic in $\mathbb{C}\setminus \gamma_1$ (resp. $\mathbb{C} \setminus \gamma_2$) and, due to the Plemelj-Sokhotzki formula (\RED{applicable since $f$ is Lipschitz continuous on $\gamma_1 \cup \gamma_2$}, see \cite{Begehr1994}), 
			\begin{align}
				&	f(\tau) = y_1(\tau^r) - y_1(\tau^l), \ \forall \tau \in \gamma_1, \label{eq.Bla1} \\
				&	f(\tau) = y_2(\tau^r) - y_2(\tau^l), \ \forall \tau \in \gamma_2.	\label{eq.Bla2}
			\end{align}
			Here, for $j=1,2$, $y_j(\tau^r)$ (resp. $y_j(\tau^l)$) refers to the limiting value of $y_j(\tau)$ as $\tau$ approaches $\gamma_j$ from the right (resp. left), as illustrated in Figure \ref{fig:RiemannSphere2}.
			We now show that $y_1$ and $y_2$ are in fact continuous everywhere in $\mathbb{C}$ thus proving the theorem. 
			Continue $\gamma_1$ (resp. $\gamma_2$) towards $0$ via curves $s_1$ (resp. $s_2$) and set $f \equiv 0$ on $s_1$ (resp. $s_2$), cf.\! Figure \ref{fig:RiemannSphere2} top right. As $\gamma_1$ and $\gamma_2$ are completely within the sector $\varphi_2<\arg(z)<\varphi_1$ (and because they have finite length within each finite disk), these curves can be chosen to lie completely within this sector as well. Denote the curve $\gamma_1\cdot s_1$  (resp. $\gamma_2 \cdot s_2$) by $\tilde{\gamma}_1$ (resp. $\tilde{\gamma}_2$).
			
			Then all previously formulated formulae still hold for our new $\tilde{\gamma}_1$ and $\tilde{\gamma}_2$, that is, upon defining 
			\begin{align*}
				&	\tilde{y}_1(z) = \frac{1}{2 \pi i} \int_{\tilde{\gamma}_1} \frac{f(z')}{z' -z} dz', \\
				&	\tilde{y}_2(z) = \frac{1}{2 \pi i} \int_{\tilde{\gamma}_2} \frac{f(z')}{z' -z} dz',
			\end{align*}
			we obtain
			\begin{align}
				&	f(\tau) = \tilde{y}_1(\tau^r) - \tilde{y}_1(\tau^l), \ \forall \tau \in \tilde{\gamma}_1, \\
				&	f(\tau) = \tilde{y}_2(\tau^r) - \tilde{y}_2(\tau^l), \ \forall \tau \in \tilde{\gamma}_2,	
			\end{align}
			since, for $j=1,2$, we have $y_j(\tau) = \tilde{y}_j(\tau)$ for $\tau \in \gamma_j$ and $\tilde{y}_j(\tau)$ is continuous on $s_j$.
			Now, define 
			\begin{align*}
				&	Y^r_1(s) = e^{i\varphi_2}\int_{0}^{\infty} \tilde{y}_1(\tau e^{i\varphi_2}) e^{is \tau e^{i\varphi_2}} d\tau, \\
				&	Y^l_1(s) = e^{i\varphi_1}\int_{0}^{\infty} \tilde{y}_1(\tau e^{i\varphi_1}) e^{is \tau e^{i\varphi_1}} d\tau, \\
				&	Y^r_2(s) =  e^{i\varphi_2}\int_{0}^{\infty} \tilde{y}_2(\tau e^{i\varphi_2}) e^{is \tau e^{i\varphi_2}} d\tau, \\
				&	Y^l_2(s) = e^{i\varphi_1}\int_{0}^{\infty} \tilde{y}_2(\tau e^{i\varphi_1}) e^{is \tau e^{i\varphi_1}} d\tau.
			\end{align*}
			These functions are, due to exponential decay, analytic	in the sectors $-\varphi_1 < \arg(s) < \pi -\varphi_1$ (for $Y^{l}_{1,2}$) and $-\varphi_2 < \arg(s) < \pi -\varphi_2$ (for $Y^{r}_{1,2}$), respectively, and have thus the common domain $-\varphi_2 < \arg(s) < \pi - \varphi_1$. By Cauchy's theorem, we first find:
			\begin{align}
				&	Y^r_1(s) = \int_{\tilde{\gamma}_1} \tilde{y}_1(\tau^r) e^{is \tau} d\tau, \\
				&	Y^l_1(s) = \int_{\tilde{\gamma}_1} \tilde{y}_1(\tau^l) e^{is \tau} d\tau, \\
				&	Y^r_2(s) = \int_{\tilde{\gamma}_2} \tilde{y}_2(\tau^r) e^{is \tau} d\tau, \\
				&	Y^l_2(s) = \int_{\tilde{\gamma}_2} \tilde{y}_2(\tau^l) e^{is \tau} d\tau. 
			\end{align}
			Then, by \eqref{eq.AppenedixIntegralUnique}, we have: 
			\begin{align}
				Y^r_1(s) +Y^r_2(s) - (Y^l_1(s) + Y^l_2(s)) \equiv 0, \ \forall s \in L.
			\end{align}
			Since two analytic functions coinciding on a line coincide on the entirety of their common domain,	we can continue $Y^r = Y^r_1 +Y^r_2$ and $Y^l=Y^l_1+Y^l_2$ to a function $Y$ analytic in the sector $-\varphi_1 < \arg(s) < \pi -\varphi_2$. The remainder of the proof is identical to \cite{AssierShanin2019}. That is, introduce the contours $\Gamma_1$, $\Gamma_2$, and $\Gamma_c$ as shown in Figure \ref{fig:GammaContours1}, 
			\begin{figure}[h!]
				\centering
				\includegraphics[width=.4\textwidth]{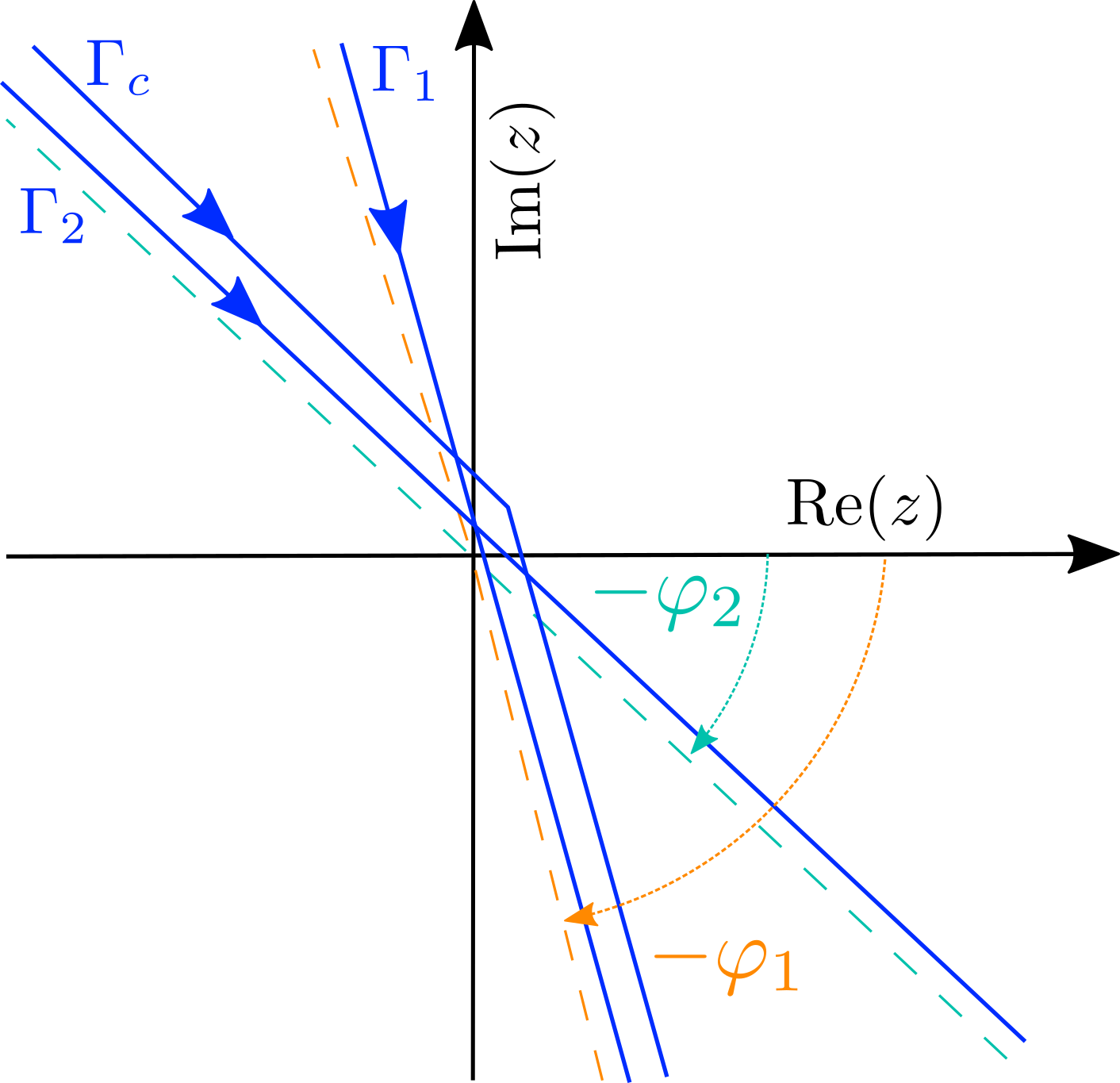}
				\caption{Contours $\Gamma_{1,2,c}$.}
				\label{fig:GammaContours1}		
			\end{figure}
			and obtain $\tilde{y}_1(\tau e^{i\varphi_1}) + \tilde{y}_2(\tau e^{i\varphi_1})$ and $\tilde{y}_1(\tau e^{i\varphi_2}) +\tilde{y}_2(\tau e^{i\varphi_2})$ by inverse Mellin transform:
			\begin{align}
				\tilde{y}_1(\tau e^{i \varphi_{j}}) + \tilde{y}_2(\tau e^{i \varphi_{j}}) = \frac{1}{2 \pi} \int_{\Gamma_{j}} Y(s) e^{is \tau e^{i\varphi_{j}}} ds, \ j=1,2.
			\end{align}
			Due to the just established analyticity properties, Cauchy's theorem and exponential decay, we find 
			\begin{align}
				\tilde{y}_1(\tau e^{i \varphi_{j}}) + \tilde{y}_2(\tau e^{i \varphi_{j}}) = \frac{1}{2 \pi} \int_{\Gamma_{c}} Y(s) e^{is \tau e^{i\varphi_{j}}} ds, \ j=1,2. \label{eq.BLA}
			\end{align}
			As in \cite{AssierShanin2019}, this yields the analytic continuation of $\tilde{y}_1 + \tilde{y}_2$ since the right hand side in \eqref{eq.BLA}  is analytic in the sector $\varphi_2 < \arg(\varphi) <\varphi_1$ showing 
			\begin{align}
				\tilde{y}_1(\tau^r) + \tilde{y}_2(\tau^r) = \tilde{y}_1(\tau^l) + \tilde{y}_2(\tau^l), \ \forall \tau \in \tilde{\gamma}_1 \cup \tilde{\gamma}_2.
			\end{align}
			But since $\tilde{y}_1$ (resp. $\tilde{y}_2$) is analytic on $\tilde{\gamma}_2$ (resp. $\tilde{\gamma}_1$), we find
			\begin{align}
				&	\tilde{y}_1(\tau^r) = \tilde{y}_1(\tau^l), \ \forall \tau \in \tilde{\gamma}_1, \\
				&	\tilde{y}_2(\tau^r) = \tilde{y}_2(\tau^l), \ \forall \tau \in \tilde{\gamma}_2,
			\end{align}	
			giving $f \equiv 0$.	
		{\hfill{}}
		\end{proof}
		
		Then, using Theorem \ref{thm.LaplaceLike}, we find:
		\begin{corollary}[2D Uniqueness Theorem]\label{thm.LaplaceLike2}
			Let $\gamma=\gamma_1 \cup \gamma_2$ and $L = L_j, \ j=1,2$ be as in Theorem \ref{thm.LaplaceLike}. Let $f: \gamma \times \gamma \to \mathbb{C}$ be Lipschitz continuous along $\gamma \times \gamma$ and let $f$ satisfy $f(z_1,z_2)= \mathcal{O}(1/|z_1|^{\beta_1}|z_2|^{\beta_2})$, $\beta_{1,2}>0$, as $|z_1|$ and/or $|z_2| \to \infty$. If 
			\begin{align}
				\int_{\gamma} \int_{\gamma} f(z_1,z_2)e^{i(z_1s_1+z_2s_2)} dz_1 dz_2 =0, \ \forall s_1 \in L_1 \text{ and } \forall s_2 \in L_2,
			\end{align}
			then 
			\begin{align}
				f \equiv 0, \ \text{on} \ \gamma\times \gamma.
			\end{align}			
		\end{corollary}
		
		The proof is identical to \cite{AssierShanin2019}, and omitted for brevity. Moreover, these results can be directly generalised to the case of $n \in \mathbb{N}$ curves $\gamma_1 \cup ... \cup \gamma_n$ by following the construction given in Theorem \ref{thm.LaplaceLike}'s proof.
		
		\section{Single integral analytical continuation formulae}\label{Appendix:Residues}

		Here, we show how \eqref{eq.PsiContFormula1} and \eqref{eq.PsiContFormula2} can be simplified, as mentioned in Remark \ref{remark:Leray}. Thereafter, we give analogous simplifications of formulae \eqref{eq.PsiCont} and \eqref{eq.PsiCont2}, thereby simplifying all formulae for analytic continuation derived in the present article. We discuss this for rewriting \eqref{eq.PsiContFormula1} only, as the procedure for rewriting \eqref{eq.PsiContFormula2}, \eqref{eq.PsiCont}, and \eqref{eq.PsiCont2} is analogous. 
		
		For simplifying \eqref{eq.PsiContFormula1}, it is sufficient to focus on the double integral 
		\begin{align*}
			J(\aalpha) &= \int_{-\infty - i \varepsilon}^{\infty - i \varepsilon} \int_{-\infty + i \varepsilon}^{\infty + i \varepsilon} \frac{K(z_1,z_2) \Psi_{++}(z_1,z_2)}{K_{\circ -}(\alpha_1,z_2)(z_2-\alpha_2)(z_1-\alpha_1)} dz_1 dz_2 \\
			& = \int_{-\infty - i \varepsilon}^{\infty - i \varepsilon} \frac{1}{K_{\circ -}(\alpha_1,z_2)(z_2-\alpha_2)} \left(\int_{-\infty + i \varepsilon}^{\infty + i \varepsilon} \frac{K(z_1,z_2) \Psi_{++}(z_1,z_2)}{(z_1-\alpha_1)} dz_1 \right) dz_2, \numberthis \label{eq.First}
		\end{align*}
		which is the integral term in \eqref{eq.PsiContFormula1}. Fix $z_2 = z_2^{\star}$ and focus only on the $z_1$ integral 
		$$\int_{-\infty + i \varepsilon}^{\infty + i \varepsilon} \frac{K(z_1,z^{\star}_2) \Psi_{++}(z_1,z^{\star}_2)}{(z_1-\alpha_1)} dz_1.$$
		Now, since $\Psi_{++}(z_1,z_2)$ is analytic within $\UHP(\RED{-2\varepsilon}) \times \UHP(\RED{-2\varepsilon})$, we know that $\Psi_{++}(z_1,z_2^{\star})$ is analytic for $z_1 \in \UHP$ and therefore, the integrand  $$\frac{K(z_1,z^{\star}_2) \Psi_{++}(z_1,z^{\star}_2)}{(z_1-\alpha_1)}$$
		has only one pole in the $z_1$ upper half plane, given by 
		\begin{align}
			z_1^{\dagger} = \sqrt[\rightarrow]{k_1^2 - (z^{\star}_2)^2}.
		\end{align}
		This is a first order pole of $K(z_1,z_2^{\star})$. Therefore, after a straightforward calculation, the residue theorem yields 
		\begin{align}
			\int_{-\infty + i \varepsilon}^{\infty + i \varepsilon} \frac{K(z_1,z^{\star}_2) \Psi_{++}(z_1,z^{\star}_2)}{(z_1-\alpha_1)} dz_1  =  -2 i \pi \frac{(k_2^2 - k_1^2) \Psi_{++}(z_1^{\dagger}, z_2^{\star})}{(z_1^{\dagger} - \alpha_1) 2 z_1^{\dagger}},
		\end{align}
		and thus \eqref{eq.PsiContFormula1} can be rewritten as
		\begin{align*}
	\Psi_{++}(\aalpha) =& \frac{-i}{4 \pi K_{\circ +}(\aalpha)} \ \int_{-\infty - i \varepsilon}^{ \infty - i \varepsilon} \frac{\left(k^2_2 - k_1^2 \right) \Psi_{++}\left(\sqrt[\rightarrow]{k^2_1 -z^2_2}, z_2\right)}{K_{\circ -}(\alpha_1,z_2)(z_2-\alpha_2) \left(\sqrt[\rightarrow]{k^2_1 -z^2_2} - \alpha_1\right)\sqrt[\rightarrow]{k^2_1 -z^2_2}}  dz_2 \\
	& -  \frac{P_{++}(\aalpha)}{K_{\circ -} (\alpha_1, \mathfrak{a}_2)K_{\circ +}(\aalpha)}.
	\numberthis \label{eq.PsiContFormula11}
	\end{align*}	 
	Similarly, \eqref{eq.PsiContFormula2} can be rewritten as
	\begin{align*}
	\Psi_{++}(\aalpha) = &  \frac{-i}{4 \pi K_{+ \circ}(\aalpha)} \ \int_{-\infty - i \varepsilon}^{ \infty - i \varepsilon} \frac{\left(k^2_2 - k_1^2 \right) \Psi_{++}\left(z_1, \sqrt[\rightarrow]{k^2_1 -z^2_1}\right)}{K_{-\circ }(z_1,\alpha_2)(z_1-\alpha_1)\left(\sqrt[\rightarrow]{k^2_1 -z^2_1} - \alpha_2\right)\sqrt[\rightarrow]{k^2_1 -z^2_1}}  dz_1  \\
	& -  \frac{P_{++}(\aalpha)}{K_{- \circ} (\mathfrak{a}_1, \alpha_2)K_{+\circ}(\aalpha)}.
	\numberthis \label{eq.PsiContFormula22}
	\end{align*}
Again, these formulae are valid for $\aalpha \in \S \times \S$, but can be used for analytical continuation similar to the procedure outlined in Section \ref{sec:AnalyticalContinuation}. Specifically, following the discussion of Section \ref{subsec:FirstStep}, we find that \eqref{eq.PsiContFormula11} yields analyticity of $\Psi_{++}$ within $(H^-\setminus \{\a_1\}) \times \UHP$ whereas \eqref{eq.PsiContFormula22} yields analyticity of $\Psi_{++}$ within $\UHP \times (H^-\setminus \{\a_2\})$.

 Formulae \eqref{eq.PsiCont} and \eqref{eq.PsiCont2} can be rewritten similarly. That is, we may either use the residue theorem in formulae \eqref{eq.PsiCont} and \eqref{eq.PsiCont2}, respectively, or we may change the contour of integration in formulae \eqref{eq.PsiContFormula11} and  \eqref{eq.PsiContFormula22}, respectively, from $(-\infty - i \varepsilon, \infty - i \varepsilon)$ to $P$. After a lengthy but straightforward calculation, this yields 
 	\begin{align*}
 	\Psi_{++}(\aalpha)=& \frac{-i}{4 \pi K_{\circ+}} \int_{P} \frac{\left(k^2_2 - k_1^2 \right) \Psi_{++}\left(\sqrt[\rightarrow]{k^2_1 -z^2_2}, z_2\right)}{K_{\circ -}(\alpha_1,z_2)(z_2-\alpha_2) \left(\sqrt[\rightarrow]{k^2_1 -z^2_2} - \alpha_1\right)\sqrt[\rightarrow]{k^2_1 -z^2_2}} dz_2 \\
 	& - \frac{K_{-\circ}(\alpha_1,\mathfrak{a}_2)}{K_{\circ+} K_{\circ-}(\alpha_1,\mathfrak{a}_2)K_{-\circ}(\mathfrak{a}_1,\mathfrak{a}_2) (\alpha_1 -\mathfrak{a}_1)(\alpha_2-\mathfrak{a}_2)}, \numberthis \\
 	\Psi_{++}(\aalpha)=& \frac{-i}{4 \pi K_{+ \circ}}  \int_{P}  \frac{\left(k^2_2 - k_1^2 \right) \Psi_{++}\left(z_1, \sqrt[\rightarrow]{k^2_1 -z^2_1}\right)}{K_{-\circ }(z_1,\alpha_2)(z_1-\alpha_1)\left(\sqrt[\rightarrow]{k^2_1 -z^2_1} - \alpha_2\right)\sqrt[\rightarrow]{k^2_1 -z^2_1}} dz_1 \\
 	& - \frac{K_{\circ-}(\mathfrak{a}_1,\alpha_2)}{K_{+\circ} K_{-\circ}(\mathfrak{a}_1,\alpha_2)K_{\circ-}(\mathfrak{a}_1,\mathfrak{a}_2) (\alpha_1 -\mathfrak{a}_1)(\alpha_2-\mathfrak{a}_2)}, \numberthis
 	\end{align*} 
 and these formulae can be used for analytic continuation of $\Phi_{3/4}$ within $(H^- \setminus \{a_1\}) \times (H^- \setminus \{a_2\})$.


\bibliographystyle{unsrt}
\bibliography{bibliography}

	\end{document}